%% file: main.tex
\documentclass{article}
\input{packages.tex}

\title{Run-and-tumble motion: the role of reversibility}
\author{Bart van Ginkel\footnote{G.J.vanGinkel@tudelft.nl} \and Bart van Gisbergen\footnote{B.L.vanGisbergen@student.tudelft.nl} \and Frank Redig\footnote{F.H.J.Redig@tudelft.nl}}
\date{%
    Delft Institute of Applied Mathematics\\
    \today
}

\begin{document}

\maketitle

\input{abstract.tex}
\input{introduction}

\input{prelim}
\input{difcoef}

\input{reversible}
\input{largedev}
\input{twostates}

\appendix
\input{appendix}

\section*{Acknowledgement}
The authors thank Richard Kraaij and Mikola Schlottke for helpful discussions.
The support of the grant 613.009.112 of the Netherlands Organisation for Scientific Research (NWO) is gratefully acknowledged.

\phantomsection
\bibliographystyle{unsrtnat}
\bibliography{refs}

\end{document}

%% file: packages.tex
\usepackage{amsmath,amsthm,mathrsfs}
\usepackage{bbm}
\usepackage{amsfonts}
\usepackage{anysize}
\usepackage{titlesec}
\usepackage{hyperref}
\usepackage[nottoc]{tocbibind}
\usepackage[titletoc]{appendix}
\usepackage[font=small,labelfont=bf]{caption}
\usepackage{tikz}
\usepackage{enumerate}
\usepackage{parskip}

\usepackage[comma,numbers]{natbib} 
\renewcommand{\cite}{\citet}
\usepackage{hyperref}
\hypersetup{ colorlinks=true, linkcolor=blue, citecolor=blue,
  filecolor=blue, urlcolor=blue}

\newtheoremstyle{mystyle}
  {}
  {}
  {}
  {}
  {}
  {.}
  { }
  {{\underline{\thmname{#1}\thmnumber{ #2}~\thmnote{(#3)}}}}

\newtheoremstyle{mystyle2}
  {}
  {}
  {\itshape}
  {}
  {}
  {.}
  { }
  {{\underline{\thmname{#1}\thmnumber{ #2}~\thmnote{(#3)}}}}

\theoremstyle{definition}
\theoremstyle{mystyle}
\newtheorem{defi}{Definition}[section]
\newtheorem{rmk}[defi]{Remark}
\newtheorem{lemma}[defi]{Lemma}
\newtheorem{example}{Example}

\theoremstyle{plain}
\theoremstyle{mystyle2}
\newtheorem{thm}[defi]{Theorem}
\newtheorem{prop}[defi]{Proposition}
\newtheorem{cor}[defi]{Corollary}

\makeatletter

\newcommand{\R}{\mathbb{R}}
\newcommand{\N}{\mathbb{N}}

\newcommand{\Z}{\mathbb{Z}}
\newcommand{\p}{\mathbb{P}}
\newcommand{\E}{\mathbb{E}}
\newcommand{\var}{\mathrm{Var}}
\newcommand{\Var}{\mathrm{Var}}

\newcommand{\e}{\mathrm{e}}
\newcommand{\dd}{\mathrm{d}}

\newcommand{\1}{\mathbbm{1}}

\newcommand{\s}{\mathscr{S}}
\newcommand{\indot}{\boldsymbol{\cdot}}
\newcommand{\epsi}{\ensuremath{\epsilon}}
\newcommand{\si}{\ensuremath{\sigma}}
\newcommand{\beq}{\begin{eqnarray}}
\newcommand{\eeq}{\end{eqnarray}}
\newcommand{\be}{\begin{equation}}
\newcommand{\ee}{\end{equation}}
\newcommand{\pee}{\ensuremath{\mathbb{P}}}
\newcommand{\br}{\begin{rmk}}
\newcommand{\er}{\end{rmk}}

\newcommand{\bt}{\begin{thm}}
\newcommand{\et}{\end{thm}}

\newcommand{\bd}{\begin{definition}}
\newcommand{\ed}{\end{definition}}

\newcommand{\bp}{\begin{prop}}
\newcommand{\ep}{\end{prop}}

\newcommand{\bc}{\begin{cor}}
\newcommand{\ec}{\end{cor}}

\newcommand{\bpr}{\begin{proof}}
\newcommand{\epr}{\end{proof}}
\newcommand{\bi}{\begin{itemize}}
\newcommand{\ei}{\end{itemize}}

\newcommand{\ben}{\begin{enumerate}}
\newcommand{\een}{\end{enumerate}}

\newcommand{\omu}{\overline{\mu}}
\newcommand{\homu}{\widehat{\mu}}
\newcommand{\caN}{{\mathcal N}}

\newcommand{\Cov}{\mathrm{Cov}}

\newcommand{\sym}{\mathrm{sym}}
\newcommand{\asym}{\mathrm{asym}}
\newcommand{\mesh}{\mathrm{mesh}}

%% file: abstract.tex
\begin{abstract}
    We study a model of active particles that perform a simple random walk and on top of that have a preferred direction determined by an internal state which is modelled by a stationary Markov process. First we calculate the limiting diffusion coefficient. Then we show that the `active part' of the diffusion coefficient is in some sense maximal for reversible state processes. Further, we obtain a large deviations principle for the active particle in terms of the large deviations rate function of the empirical process corresponding to the state process. Again we show that the rate function and free energy function are (pointwise) optimal for reversible state processes. Finally, we show that in the case with two states, the Fourier-Laplace transform of the distribution, the moment generating function and the free energy function can be computed explicitly. Along the way we provide several examples.
\end{abstract}

%% file: introduction.tex
\section{Introduction}\label{sec:introduction}
In this paper we study run-and-tumble motion, which is often used as a model of active particles.
The particle motion has two ingredients: first the particle performs a symmetric random walk, and second, independently it moves in a direction dictated by an internal state process.
This internal state process is assumed to be a continuous-time stationary Markov process.
In the sequel we will first describe how our paper relates to various results on run-and-tumble particles in the literature. Next, we will briefly sketch how our model relates to the broader literature on active matter, stochastic slow-fast systems and directionally reinforced random walks.

\subsection{Model and contributions of this paper}
The model that we study in this paper is an instance of what is more generally called run-and-tumble motion. These are models of particles that follow a preferred direction which is reversed at random points in time. Recent articles include \cite{grossmann2016diffusion}, \cite{demaerel2018active}, \cite{malakar2018steady}, \cite{le2019noncrossing}, \cite{dhar2019run} and \cite{garcia2020run}.

We study an active particle of which the state process (that determines the preferred direction) is a stationary Markov process (under some technical assumptions), started from its unique ergodic measure. Then our main contribution is twofold. First we are able to calculate closed form formulas for the limiting diffusion coefficient of the active particle. This formula holds in great generality, including also the case where the state process is a diffusion (we will provide examples where an Ornstein-Uhlenbeck process or Brownian motion on a circle form the state process). In this formula we can interpret the different terms and observe where the activity is manifested. We also calculate the large deviations free energy function and rate function in the case where the state process has a finite state space.

Second, we study the role of reversibility of the state process in the diffusion coefficient and large deviations of the active particle (again for finite state spaces). In particular, we show that reversible processes in some sense optimize those quantities. To be more precise, we show that among all processes with the same symmetric part and the same stationary measure, the reversible process maximizes the diffusion coefficient and the free energy function (pointwise) and minimizes the large deviations rate function (also pointwise). The last two results are obtained by showing a pointwise inequality for the Donsker-Varadhan rate function of the empirical processes corresponding to the reversible and non-reversible state processes, respectively.

The calculations that we present are for an active particle in $\R$, but we explain for all of our results how they generalize to $\R^d$ and we also provide the explicit formulas in the $\R^d$ setting.

\subsection{Context and related literature}
First of all, the run-and-tumble motion is often used as a model of active matter. 
As we said before, our active particle performs a symmetric random walk and a random walk with preferred directions that are switched. The part of the motion that follows the internal state is called the active part of the motion, because for the switching between internal states some internal source of energy is needed. The passive part of the motion is the symmetric random walk part and comes from collisions with surrounding molecules.

Note that active particles should not be confused with activated random walk. In those models particles perform random walks, but fall asleep after a random time and are awakened (activated) when other particles jump to their position.

Second, the active particle motion studied in this paper is an example of a stochastic slow-fast system. These are well-studied systems where coupled quantities evolve on different time scales. If one rescales the position of the active particle diffusively, the underlying state process behaves as a fast process and the (rescaled) particle position is a slow process. Asymptotically the fast state process averages out and has a deterministic influence on the slow process: the limiting diffusion coefficient will depend on the state process only through the stationary distribution and the covariance function. For an introduction to stochastic slow-fast systems see for instance~\cite{berglund2006noise}. The large deviation results that we obtain are related to more general results for large deviations of slow-fast systems that were studied in for instance~\cite{freidlin1998random} or, more recently, in~\cite{kraaij2020large}.

Third, the active particle motion studied in this paper has strong similarities with a directionally reinforced random walk. This model was first studied by~\cite{coppersmith1987random} and a multidimensional version in~\cite{mauldin1996directionally}. Then in \cite{horvath1998limit} and (in a more general context)~\cite{ghosh2014directionally} it was shown for a process of this type that it converges to a multidimensional Brownian motion when rescaled diffusively.

Also we will compare the diffusion coefficients and large deviations rate functions for active particles with state processes that are either reversible or non-reversible with respect to the same invariant measure. In particular we will show that the Donsker-Varadhan rate function of reversible processes is dominated by the rate functions of non-reversible processes with the same symmetric part and the same invariant measure. A similar result (in a different context) was obtained in~\cite{pinsky1985function}.

\subsection{Structure of this paper}
 In Section~\ref{sec:prelim}, we introduce the active particle process as a stochastic integral. We split it into a random walk part, a martingale part and an active part.

In Section~\ref{sec:calcdifcoef}, we obtain the limiting diffusion coefficient of the active particle and show that it is the sum of the contributions of the random walk part, the martingale part and the active part. Then we generalize the formulas to the multidimensional case. The limiting diffusion coefficient (or matrix) is then calculated for several concrete examples, both with finite and with infinite state spaces.  Finally, we sketch how one obtains a Central Limit Theorem for the active particle.

Next, in Section~\ref{sec:revopt}, we restrict ourselves to finite space spaces and study the active part of the diffusion coefficient, which is proportional to an inner product with the inverse of the generator of the state process. We show that among all stationary processes with respect to the same invariant measure and with equal symmetric part, the active part of the diffusion coefficient is maximal for the reversible process. We use the 1-dimensional case to show that this also holds for the active part of the diffusion matrix in higher dimensions.

Then in Section~\ref{sec:largedev}, we move to large deviations (still for finite state spaces). We compute the large deviations free energy function. Using Varadhan's lemma, we derive an expression for the free energy function of the active particle in terms of the Donsker-Varadhan rate function for the empirical process corresponding to the state process (which in turn gives us the large deviations rate function as the Legrende transform of the free energy). We show that the free energy function is maximal and the rate function is minimal in the reversible case (similar to the situation for the diffusion coefficient) by showing that Donsker-Varadhan rate function is maximal for reversible processes.

We conclude the paper in Section~\ref{sec:twostates} with an analysis of the situation where the state space is $\{\pm1\}$. In this two-state case we can explicitly calculate the Fourier-Laplace transform of the distribution of the active particle process, the moment generating function and the large deviations free energy function.

%% file: prelim.tex
\section{Preliminaries}\label{sec:prelim}

In this section we introduce the model and goal of this paper. First, in Section~\ref{subsec:informal} we describe in words the models we study and formulate in words the main results. Then in Section~\ref{subsec:mathematical}, the definitions will be repeated with more mathematical details and precise assumptions. In Section~\ref{subsec:informal} we will also describe the basic example where the internal state space is $\{-1,1\}$, more examples will follow in Section~\ref{subsec:examples}.

\subsection{Informal description of the model and main results}\label{subsec:informal}
In the models we consider a particle that moves on $\R^d$ in continuous time. The particle has a position at time $t\geq 0$ denoted by $X_t$, and an ``internal state'' denoted by $M_t$. The internal state is assumed to evolve according to a stationary Markov process, and  can model e.g. a chemical state of a molecular motor. The active part of the motion is driven by this internal state. The simplest setting is e.g. when the internal state takes the values $\pm 1$ and determines whether the particle drifts to the right or left.

Let us now first describe the joint motion of the position and the internal state $(x,m)$ in the simplest setting where the particle moves on the discrete lattice $\Z$ and has internal state $m\in \{-1,1\}$. The motion consists of three parts.
\begin{itemize}
    \item[a)] At rate $\kappa$ the particle makes a random walk jump, i.e., $(x,m)$ moves to $(x\pm 1, m)$. This motion models the ``passive'' part of the motion, caused by collisions with surrounding molecules.
    \item[b)] At rate $\lambda$ the particle jumps according to its internal state, i.e., $(x,m)$ moves to $(x+m,m)$. This corresponds to the active part of the motion, driven by an internal energy source (such as ATP-ADP conversion).
    \item[c)] At rate $\gamma$ the internal state flips, i.e., $(x,m)$ moves to $(x,-m)$.
\end{itemize}

Denoting $\mu(x,m;t)$ the probability to be at position $x\in \Z$, and internal state $m\in \{-1,1\}$ the above verbal description  of the process is then summarised via the master equation
\begin{eqnarray}
    \frac{d\mu(x,m;t)}{dt} &= &\kappa (\mu(x+1,m;t)+\mu(x-1,m;t)-2\mu(x,m;t)) \nonumber\\
    &+& \lambda (\mu(x-m,m;t)-\mu(x,m;t)) + \gamma(\mu(x,-m;t)-\mu(x,m;t))\nonumber
\end{eqnarray}
or alternatively via the generator working on functions from the state space $\Z\times \{-1,1\}$
\begin{eqnarray}
    L f(x,m) &=& \kappa (f(x+1,m) + f(x-1,m)-2f(x,m))\nonumber\\
    &+& \lambda (f(x+m,m)-f(x,m)) +\gamma( f(x,-m)- f(x,m)).\nonumber
\end{eqnarray}
The idea is now to generalise this simple setting, i.e, the motion of the particle is on $\R^d$ and  we will allow much more general internal state processes (the precise assumptions on them are in the subsection below) including e.g. diffusion processes such as the Ornstein-Uhlenbeck process. 
In this more general setting, the active part of the motion consists in jumping according to the vector $v(m)$ determined by the internal state $m$, and the internal state is a general stationary ergodic Markov process, whereas the random walk part of the motion remains unchanged. 
This implies that the generator is of the form
\begin{eqnarray}
    L f(x,m) &=& \kappa (f(x+1,m) + f(x-1,m)-2f(x,m))\nonumber\\
    &+& \lambda (f(x+ v(m),m)-f(x,m)) +\gamma Af(x, \cdot)(m)\nonumber
\end{eqnarray}

where $A$ is the generator of the internal state process. Notice that this form of the generator implicitly assumes that the internal state dynamics is not depending on the particle's position. Moreover, we assume that there is no ``global'' drift in the active part of the motion, i.e., the average of $v(m)$ over the stationary distribution of the internal state process is assumed to be zero. Note that the active particle with internal state space $\{-1,1\}$ in the simple setting above fits into this framework by letting $v$ be the identity function on $\{-1,1\}$.

Our main interest is then in the asymptotic behavior of the position $X_t$, more precisely we will prove the following:
\begin{enumerate}
    \item Diffusive scaling limit for $X_t$, with explicit expressions for the diffusion matrix $D$, i.e., in the scaling limit
    \[
    \frac{1}{\sqrt N} X_{tN}\to  \sqrt{D}W(t), \hspace{5mm} N\to\infty
    \]
    where $W$ denotes Brownian motion, and where $D$ denotes the diffusion matrix.
    \item Large deviations for the position $X_t, t\geq 0$, i.e., in the sense of large deviations
    \[
    \pee \left(\frac{X_t}{t} \approx x\right) \approx e^{-t I(x)}
    \]
    with $I(x)$ the large deviation rate function.
\end{enumerate}

We then focus on the question how  the diffusion matrix as well as the large deviation rate depend on the internal state process, more precisely on its generator $A$. We show that both quantities are optimised (i.e., the diffusion matrix is maximal and the rate function is minimal) for reversible internal state space processes.

More precisely, when the stationary distribution $\mu$ of the internal state process is fixed, as well the reversible part of the dynamics, then we show that the diffusion matrix is maximal as the rate function is minimal for the internal state process for which $\mu$ is reversible, i.e., when the asymmetric part of the dynamics is zero. Though we do not have a simple intuitive ``physics'' argument for this result, it corresponds to the general intuition that non-reversible processes converge faster to their stationary state, and therefore allow less fluctuations, resulting in a smaller rate function (and larger diffusion constant) in the reversible setting.

\subsection{Mathematical definitions}\label{subsec:mathematical}

We consider the position $(X_t,t\geq 0)$ of a particle that moves in continuous time and space (see also Remark~\ref{rem:Xcont}). For now we assume $X_t\in \R$, but we will generalize to $\R^d$ later. The particle has the following dynamics.
\begin{itemize}
    \item[a)] With rate $2\kappa$ the particle performs a simple symmetric random walk.
    \item[b)] Independently, with rate $\lambda$ the particle jumps in a preferred direction indicated by an inner state. If such jump occurs at time $t$, the particle jumps from $X_t$ to $X_t+v_t^\gamma$.
    \item[c)] This internal state evolves with `rate' $\gamma$ according to a stationary Markov process.
\end{itemize}
Because of the jump to a preferred direction based on the inner state, we call the particle an \textit{active particle}.

To make this more precise we make the following definitions. We will assume that the processes in the coming definitions are jointly defined on a probability space $(\Omega,\mathscr{F},\p)$. 
\begin{itemize}
    \item[i)] Random walk part. Let $Y=(Y_t,t\geq 0)$ be a simple symmetric random walk, i.e. a random walk that starts from the origin ($Y_0=0$), jumps with rate $1$ and jumps $1$ to the left or to the right with equal probability. Fix a constant $\kappa>0$. Then the random walk part of the process is $Y_{2\kappa t}$.
    \item[ii)] Internal state process. Let $(M_t,t\geq 0)$ be a stationary Markov process (independent of the random walk) on a state space $\s$ with ergodic measure $\mu$. We will call this process the \textit{state process}. Since we will always start $M$ from $\mu$, we can assume without loss of generality that $\mu$ is the unique ergodic (and hence the unique invariant) measure of $M$. Denote by $(S_t,t\geq 0)$ and $A$ the corresponding semigroup and Markov generator on $L^2(\mu)$, respectively, and denote the inner product on $L^2(\mu)$ by $(\cdot,\cdot)$ and the corresponding norm by $\|\cdot\|$.
    \item[iii)] Speed function. Let $v$ be an element of $L^2(\mu)$. We will call $v$ the \textit{speed function}. For simplicity, we assume that $\int v\dd \mu=0$, meaning that the average of the speed with respect to the stationary measure on the internal state space is $0$. This is not essential though, we will make some remarks on what happens without this assumption. The idea is that $v:\s\rightarrow \R$ is a mapping that indicates for each internal state the jump vector in case of an active jump when the particle has that internal state. In the example in Section~\ref{subsec:informal}, the speed function $v$ was just the identity function on $\{-1,1\}$. In Section~\ref{subsec:examples} we will see more examples, for instance where $v$ maps three internal states to three numbers that sum to $0$ (in Example~\ref{ex:threestates}) or where $v$ is the sine function (in Example~\ref{ex:sinBt}). 
    \item[iv)] Speed process. Fix a constant $\gamma>0$. We define $v^\gamma_t = v(M_{\gamma t})$ and call $(v^\gamma_t,t\geq 0)$ the \textit{speed process}. Note that this speed process does not need to be a Markov process. In the special case for $\gamma=1$, we will simply write $v_t$. Note that $(v^\gamma_t,t\geq 0)$ is the process $(v_t,t \geq 0)$ speeded up by the factor $\gamma$. We make the following two technical assumptions on the speed process. 
    \begin{itemize}
        \item[a)] First we assume that
\begin{equation}\label{eq:assumptionlimst}
    \lim_{t\rightarrow\infty} \int_0^t S_rv\dd r \hspace{0.5cm}\text{exists in } L^2(\mu).
\end{equation}
This implies that the limit $u:=\int_0^\infty S_tv\dd t$ satisfies $u\in D(A)$ and $-A u = v$, so we will write $\int_0^\infty S_tv\dd t = -A^{-1}v$. We need this assumption to ensure that the limiting variance is finite. If it does not hold, there may not be a diffusive scaling limit. Sufficient conditions for Assumption~\eqref{eq:assumptionlimst} are for instance that the spectral gap of $A$ is positive or that there exist $c,C>0$ such that
\begin{equation*}
    \|S_tv\|\leq C \e^{-ct}.
\end{equation*}
The latter is a condition on the speed of relaxation, it ensures that the internal state process reaches equilibrium fast enough, which avoids large temporal covariances.
In any case, Assumption~\eqref{eq:assumptionlimst} requires that $S_tv$ goes to $0$ fast enough that it is integrable.
    \item[b)] The second assumption is that for all $t>0$
\begin{equation}\label{eq:assumptioncontL2}
    \lim_{\delta\downarrow 0} \sup_{\substack{0\leq s,s'\leq t\\ |s-s'|<\delta}} \E[(v_s-v_{s'})^2] = 0.
\end{equation}
In other words: the speed process must be uniformly continuous in $L^2$. This assumption is purely technical, we will use it in Lemma~\ref{lemma:L2approxintegral} to show that the integral in~\eqref{eq:Xtdefinition} is well-defined.
    \end{itemize}
Both of these assumptions are automatically satisfied in the case that the state space $\s$ of $M$ is finite. Other internal state processes that satisfy these assumptions (with a suitable choice of $v$) include diffusion processes such as Brownian motion and the Ornstein-Uhlenbeck processes that we encounter in the examples in Section~\ref{subsec:examples}.

\item[v)] Active jumps. Finally, fix a constant $\lambda>0$ and let $(N_t,t\geq 0)$ be a Poisson process with rate $\lambda$ (independent of the random walk and the state process). This process marks the times at which the particle jumps in a preferred direction.
\end{itemize}

With these components we can define
\begin{equation}\label{eq:Xtdefinition}
    X_t = Y_{2\kappa t} + \int_0^t v_s^\gamma \dd N_s,
\end{equation}
where the integral is defined as a limit in $L^2(\p)$ (see in Lemma~\ref{lemma:L2approxintegral} how the well-definedness of the integral follows from Assumption~\eqref{eq:assumptioncontL2}). This expression matches with our description above: $Y_{2\kappa t}$ is the random walk part and on top of that whenever the Poisson process $N$ has a jump at time $t$, say, the number $v_t^\gamma$ is added to $X_t$. Note that~\eqref{eq:Xtdefinition} implies that $X_0=0$. Also, we can write~\eqref{eq:Xtdefinition} as
\begin{equation}\label{eq:threeparts}
    X_t = Y_{2\kappa t} + \int_0^t v_s^\gamma \dd \overline{N}_s + \lambda \int_0^t v_s^\gamma \dd s,
\end{equation}
where $\overline N_t = N_t-\lambda t$ is a compensated Poisson process.
We call the first, second and third term of~\eqref{eq:threeparts} the \textit{random walk part}, the \textit{martingale part} and the \textit{active part}, respectively. This division will become more clearly visible in the diffusion coefficient.

\begin{rmk}\label{rem:Xcont}
Note that if $v$ is integer-valued, $X_t$ stays in the lattice $\Z$. In case $v$ is not integer-valued, we can also directly consider a continuous process and define
\begin{equation}
    X^c_t = B_{2\kappa t} + \lambda \int_0^t v_s^\gamma \dd s,
\end{equation}
where $(B_t,t\geq 0)$ is Brownian motion (independent of the state process) and where the speed process is followed continuously in time. As will become clear later, the change to Brownian motion is mostly aesthetic. However, the change from $\dd N_t$ to $\lambda \dd t$ leaves out the martingale part of $X_t$, which will have consequences for both the limiting diffusion coefficient and for the large deviations. We will makes remarks on this later, after the results concerned.
\end{rmk}

%% file: difcoef.tex
\section{Diffusion coefficient}\label{sec:calcdifcoef}
A first observation is that the expectation of $X_t$ is $0$. Indeed, using independence of the processes $v^\gamma_s$ and $N_s$ and the fact that $\E v^\gamma_s=0$, we compute
\begin{equation*}
    \E X_t = \E Y_{2\kappa t} + \E \int_0^t v_s^\gamma \dd N_s = 0 + \lim_{n\rightarrow\infty} \sum_{i=0}^{n-1} \E [v_s^\gamma (N_{s_{i+1}}-N_{s_i})] = \lim_{n\rightarrow\infty} \sum_{i=0}^{n-1} \E v^\gamma_{s_i} \lambda(s_{i+1}-s_i) = 0.
\end{equation*}
In this section we determine the limiting diffusion coefficient of the active particle and extend this result to active particles in higher dimensions. Then we provide some examples. Finally, we discuss the invariance principle.

\subsection{Calculating the diffusion coefficient}
As a first result, we compute the limiting variance of the position of the active particle.

\subsubsection{The \texorpdfstring{$1$}{1}-dimensional case}
We start in dimension $1$. Recall that $(\cdot,\cdot)$ denotes the inner product on $L^2(\mu)$.
\begin{thm}\label{thm:limdifcoef}
The active particle has the following limiting diffusion coefficient
\begin{equation}\label{eq:difcoef}
    \lim_{t\rightarrow\infty}\frac{\var(X_t)}{t} = 2\kappa + \lambda\int v^2\dd\mu + \frac{2\lambda^2}{\gamma}(v,-A^{-1}v).
\end{equation}
\end{thm}
\begin{proof}
First of all, note that the random walk part of $X_t$ is independent of the rest. Second, note that using Lemma~\ref{lemma:L2approxintegral} and the independence of $v^\gamma$ and $\overline{N}$,
\begin{eqnarray*}
    \Cov\left(\int_0^t v_s^\gamma \dd \overline{N}_s,\lambda \int_0^t v_s^\gamma \dd s\right) &=&  \lim_{n\rightarrow\infty} \sum_{i,j=0}^{n-1} \Cov\left(v_{s_i}(\overline{N}_{s_{i+1}}-\overline{N}_{s_i}),\lambda v_{s_j} (s_{j+1}-s_j)\right)\\
    &=& \lim_{n\rightarrow\infty} \sum_{i,j=0}^{n-1} \lambda v_{s_j} (s_{j+1}-s_j) \Cov\left(v_{s_i},v_{s_j}\right)\E\left[\overline{N}_{s_{i+1}}-\overline{N}_{s_i}\right] = 0.
\end{eqnarray*}
This implies that 
\begin{equation*}
    \var(X_t) = \var(Y_{2\kappa t})+ \var\left(\int_0^t v_s^\gamma \dd \overline{N}_s\right) + \var\left(\lambda \int_0^t v_s^\gamma \dd s\right).
\end{equation*}
In other words, each of the parts of $X_t$ in~\eqref{eq:threeparts} has its own contribution to the variance of $X_t$ and hence to the limiting diffusion coefficient. Similar to before, we will refer to these as the random walk part, the martingale part and the active part of the diffusion coefficient. We will now calculate these contributions.

First, $Y_{2\kappa t}$ is the difference of two independent Poisson random variables with rate $\kappa t$. Therefore
\begin{equation}\label{eq:difcoefrandpart}
     \lim_{t\rightarrow\infty} \frac{\var(Y_t)}{t} =  \lim_{t\rightarrow\infty} \frac{\kappa t+ \kappa t}{t} = 2\kappa.
\end{equation}

Second, using Lemma~\ref{lemma:L2approxintegral}, the independence of $v^\gamma$ and $\overline{N}$ and the fact that $\E v^\gamma_s = \E\left[ \overline{N}_{s_{i+1}}-\overline{N}_{s_i}\right] = 0$, we see
\begin{eqnarray*}
    \var\left(\int_0^t v_s^\gamma \dd \overline{N}_s\right) &=& \lim_{n\rightarrow\infty} \sum_{i,j=0}^{n-1} \Cov\left(v_{s_i} (\overline{N}_{s_{i+1}}-\overline{N}_{s_i}), v_{s_j} (\overline{N}_{s_{j+1}}-\overline{N}_{s_j})\right) \\
    &=& \lim_{n\rightarrow\infty} \sum_{i=0}^{n-1} \var\left(v_{s_i} (\overline{N}_{s_{i+1}}-\overline{N}_{s_i})\right) 
    = \lim_{n\rightarrow\infty} \sum_{i=0}^{n-1} \var\left(v_{s_i}\right)\var\left(\overline{N}_{s_{i+1}}-\overline{N}_{s_i}\right)  \\
    &=& \lim_{n\rightarrow\infty} \sum_{i=0}^{n-1} \int v^2
    \dd\mu \lambda (s_{i+1}-s_i) = \lambda t \int v^2
    \dd\mu.
\end{eqnarray*}
Therefore
\begin{equation}\label{eq:difcoefmartpart}
    \lim_{t\rightarrow\infty} \frac{\var\left(\int_0^t v_s^\gamma \dd \overline{N}_s\right)}{t} =  \lim_{t\rightarrow\infty} \frac{\lambda t \int v^2
    \dd\mu}{t} = \lambda \int v^2
    \dd\mu.
\end{equation}

For the third part we calculate the limiting variance of an additive functional of a Markov process. This formula was already obtained for instance in~\cite[Corollary 1.9]{kipnis1986central} and~\cite[Lemma 2.4]{demasi1989invariance} (for reversible Markov processes). In fact, it is known as Green-Kubo relations, which go back to~\cite{green1954markoff} and~\cite{kubo1957statistical}. For completeness, we provide the calculations for our specific context here. Using the stationarity of $v^\gamma$ and the symmetry of covariance, we compute
\begin{eqnarray}
    \var\left(\int_0^t v_s^\gamma \dd s\right) &=& \int_0^t \int_0^t \Cov(v^\gamma_s,v^\gamma_r)\dd r \dd s = 2 \int_0^t \int_0^s \Cov(v^\gamma_s,v^\gamma_r)\dd r \dd s = 2 \int_0^t \int_0^s \Cov(v^\gamma_{s-r},v^\gamma_0)\dd r \dd s \nonumber \\
    &=& 2 \int_0^t \int_0^s \Cov(v^\gamma_r,v^\gamma_0)\dd r \dd s = 2 \int_0^t \int_r^t \Cov(v^\gamma_r,v^\gamma_0)\dd s \dd r = 2 \int_0^t (t-r) \Cov(v^\gamma_r,v^\gamma_0) \dd r\nonumber\\
    &=& \frac{2}{\gamma} \int_0^{\gamma t}(t-r)\Cov(v(M_{r}),v(M_0))\dd r = \frac{2}{\gamma} \int_0^{\gamma t}(t-r)(v,S_rv)\dd r.\label{eq:varthird}
\end{eqnarray}
To compute this, first note that with Assumption~\eqref{eq:assumptionlimst} we see that
\begin{equation}\label{eq:greenconv}
    \lim_{t\rightarrow\infty} \int_0^{t} (v,S_rv)\dd r = \left(v,\lim_{t\rightarrow\infty} \int_0^t S_r v\dd r\right) = \left(v,-A^{-1}v\right).
\end{equation}
Note that the convergence of $\int_0^t (v,S_rv)\dd r$ also implies that
\begin{equation}\label{eq:cesaroint}
    \lim_{t\rightarrow\infty} \int_0^{t}\frac{r}{t} (v,S_rv)\dd r = 0.
\end{equation}
Combining~\eqref{eq:varthird},~\eqref{eq:greenconv} and~\eqref{eq:cesaroint}, we obtain
\begin{equation}\label{eq:difcoefactpart}
    \lim_{t\rightarrow\infty} \frac{\var\left(\lambda \int_0^t v_s^\gamma \dd \overline{N}_s\right)}{t} = \lim_{t\rightarrow\infty} \frac{2\lambda^2}{\gamma} \int_0^{\gamma t}(v,S_rv)\dd r + \lim_{t\rightarrow\infty}2\lambda^2 \int_0^{\gamma t} \frac{r}{\gamma t} (v,S_rv)\dd r = \frac{2\lambda^2}{\gamma}(v,-A^{-1}v).
\end{equation}
Now combining~\eqref{eq:difcoefrandpart},~\eqref{eq:difcoefmartpart} and~\eqref{eq:difcoefactpart}, we obtain the result.
\end{proof}

\subsubsection{Higher dimensions}
So far we considered an active particle that only moves in one dimension. However, we can just as well treat a higher dimensional situation. To this end fix a dimension $d\in\N$. Let $Y$ be a $d$-dimensional simple random walk, i.e. each component of $Y$ is an independent copy of the $Y$ that we had in the $1$-dimensional situation. Let the speed function $v$ be an element of $L^2((\Omega,\mu),\R^d)$ such that $\int v\dd \mu=0$ (in $\R^d$). We denote by $\Sigma$ the covariance matrix of $v$ under $\mu$, i.e.
\begin{equation*}
    \Sigma_{ij} = \Cov(v(M_0)_i,v(M_0)_j).
\end{equation*}
Let again $X_t$ denote the position of the active particle, now in $\R^d$, with random walk part $Y$ and speed function $v$. The internal state process remains the same as the $1$-dimensional case. To find the limiting diffusion matrix of the active particle, we can show that similar to the $1$-dimensional case
\begin{equation*}
    \Cov((X_t)_i,(X_t)_j) = \Cov((Y_{2\kappa t})_i,(Y_{2\kappa t})_j) + \Cov\left(\int_0^t(v_s^\gamma)_i\dd \overline N_s,\int_0^t(v_s^\gamma)_j\dd \overline N_s\right) + \Cov\left(\int_0^t(v_s^\gamma)_i\dd s,\int_0^t(v_s^\gamma)_j\dd s\right).
\end{equation*}

Now if we go through calculations that are very similar to the 1-dimensional case, we obtain the following.
\begin{thm}\label{thm:ddimcase}
    Let $X_t$ be the position in $\R^d$ of the active particle that we just defined. Then
    \begin{equation}\label{eq:difmat}
        \lim_{t\rightarrow\infty}\frac{\Cov((X_{t})_i,(X_{t})_j)}{t} = 2\kappa\delta_{i,j} + \lambda \Sigma_{ij} + \frac{\lambda^2}{\gamma} [((v)_i,-A^{-1}(v)_j)+((v)_j,-A^{-1}(v)_i)].
    \end{equation}
\end{thm}
\begin{rmk}
The sum of inner products in~\eqref{eq:difmat} equals $2((v)_i,-\sym(A^{-1})(v)_j)$ (where for an operator $B$ on $L^2(\mu)$, $\sym(B) = (B+B^*)/2$ is the symmetric part). Note that the $1$-dimensional case can be retrieved from this by realising that for any operator $B$ and function $w$, $(w,Bw) = (w,\sym(B)w)$.
\end{rmk}

\subsubsection{Interpretation}
We now briefly discuss the various terms appearing in the RHS of~\eqref{eq:difcoef}. First of all, as is clear directly from the definition of the process, the random walk part is independent of the rest and therefore produces the term $2\kappa$. 

Now, to understand the other two terms, let us first consider what happens in the limit of $\gamma$ to infinity. In that case the state process is speeded up so much that it reaches equilibrium between subsequent jumps of the $N$-process. Therefore the jump sizes are just independent copies of $v(M_0)$ (so $v$ under the stationary measure $\mu$), so the process is simply a random walk with jump rate $\lambda$ and jump size distribution $v(M_0)$. In this case the diffusion coefficient should be $\lambda \Var(v(M_0))=\lambda\int v^2\dd\mu$, which is indeed what we find when we let $\gamma$ go to infinity in~\eqref{eq:difcoef}.

Finally, the third term of~\eqref{eq:difcoef} corresponds to the case where $\gamma$ is finite. Therefore this term comes from the dependence between the active jumps due to the temporal dependence in the state process. Hence this term comes from the activity of the particle. These considerations justify the name `active part' for the third part of~\eqref{eq:difcoef}. This is the only part that depends on the state process through more than just its stationary distribution. We will analyse this term more thoroughly in Section~\ref{sec:revopt}.

\begin{rmk}
Note that for $X^c$ (see Remark~\ref{rem:Xcont}), the random walk part of $X^c_t$ has variance $2\kappa t$, the martingale part is left out and the active part is the same as in $X$, so we obtain
\begin{equation*}
    \lim_{t\rightarrow\infty}\frac{\var(X^c_t)}{t} = 2\kappa+ \frac{2\lambda^2}{\gamma}(v,-A^{-1}v).
\end{equation*}
\end{rmk}
\begin{rmk}\label{rem:intcov}
Note that instead of writing $(v,-A^{-1}v)$, we could also have kept the covariance in the expression in~\eqref{eq:varthird} to obtain in a similar way that the active part of the limiting diffusion coefficient equals
\begin{equation*}
    \frac{2\lambda^2}{\gamma} \int_0^\infty \Cov(v_0,v_r)\dd r.
\end{equation*}
This might be easier to calculate for processes of which the covariance function is explicitly known.
\end{rmk}

\begin{rmk}\label{rem:zeroavg}
The assumption that $\int v\dd \mu=0$ makes sure that $\E X_t=0$. Considering a speed function that does not have average $0$ is equivalent to setting the speed function to be $v+c$ where $c$ is a constant and $v$ still satisfies $\int v\dd\mu=0$. In this case the expectation equals $\E X_t = c\lambda t$.
Of course the random walk part is not affected by this choice. Now it is easy to see following our calculations above that with the new speed function the expectation of the martingale part remains the same, but the variance changes. Contrarily, the expectation of the active part changes, but the variance stays the same (since the change is deterministic). Overall, the limiting diffusion coefficient becomes:
\begin{equation*}
    \lim_{t\rightarrow\infty}\frac{\var(X_t)}{t} = 2\kappa + \lambda\left(\int v^2\dd\mu+c^2\right) + \frac{2\lambda^2}{\gamma}(v,-A^{-1}v).
\end{equation*}
\end{rmk}

\subsection{Examples}\label{subsec:examples}
Now we give some examples. We start with two cases where the state process $M$ is a Markov chain with 2 or 3 states. Then we take $M$ to be an Ornstein-Uhlenbeck process and Brownian motion on a circle and finally we consider an Ornstein-Uhlenbeck process in $\R^2$.

First, in these examples we need to calculate $(v,-A^{-1}v)$ (cf.~\eqref{eq:difcoef}). Now write $u=-A^{-1}v$ and recall that this means $u=\int_0^\infty S_t v \dd t$, which implies $-Au=v$. In order to compute $(v,-A^{-1}v)$, we can procede as follows. First we find a function $w$ such that $-Aw=v$. Then $(v,w)=(v,-A^{-1}v)$. Indeed, since $\mu$ is the unique ergodic measure, the only $h\in D(A)$ with $Ah=0$ are constant functions, so if $-Au=-Aw$, $u$ and $w$ only differ by a constant. Therefore $(v,w) = (v,u+c\1) = (v,-A^{-1}v)+c \int v\dd \mu = (v,-A^{-1}v)$.

Second, for all of the examples, we need to verify Assumptions~\eqref{eq:assumptionlimst} and~\eqref{eq:assumptioncontL2}. In Example~\ref{ex:twostates} and~\ref{ex:threestates}, the state space is finite so both assumptions always hold. In Example~\ref{ex:OU},~\ref{ex:sinBt} and~\ref{ex:multidimOU}, Assumption~\eqref{eq:assumptioncontL2} can be verified by a direct computation, since the correlation functions for Brownian motion and the Ornstein-Uhlenbeck process are explicitly known. As we noted before, for Assumption~\eqref{eq:assumptionlimst}, it suffices to find constants $c,C>0$ such that $\|S_tv\|\leq C\exp(-ct)$. This is implied by the Poincar\'e inequality (see~\cite[Thm 2.18]{handel2016prob}). The Poincar\'e inquality for the Ornstein-Uhlenbeck process is proved in~\cite[Lem 2.22, Thm 2.25]{handel2016prob} (and holds similarly in the higher dimensional case). By~\cite[Rem 2.19]{handel2016prob}, the Poincar\'e inequality for Brownian motion with drift on the circle follows from the Poincar\'e inequality for driftless Brownian motion on the circle. The exponential ergodicity (and the corresponding Poincar\'e inequality) in this case is known and can be shown using Fourier analysis.

\begin{example}[2 states]\label{ex:twostates}
We start with the case where $M$ is a Markov chain on $\s=\{ 1,-1\}$ where the state switches with rate $1$ and $v$ is the identity function $[1,-1]^T$. Then $\mu=(\delta_{-1}+\delta_1)/2$,
\begin{equation*}
    A = \begin{bmatrix} 
        -1 & 1 \\
        1 & -1 
    \end{bmatrix}
\end{equation*}
and indeed $\int v\dd \mu =0$. Now choose $w=[1,0]^T$, then $-Aw=v$. So $(v,-A^{-1}v)=(v,w)=1/2 * (1*1)+ 1/2 * (-1*0) = 1/2$. Also we compute $\int v^2\dd \mu = \int 1 \dd \mu = 1$. Now applying Theorem~\ref{thm:limdifcoef} yields
\begin{equation*}
    \lim_{t\rightarrow \infty}\frac{\Var(X_{t})}{t}  = 2\kappa + \lambda \int v^2\dd \mu + \frac{2\lambda^2}{\gamma} (v,w) = 2\kappa + \lambda + \frac{\lambda^2}{\gamma}.
\end{equation*}
Note that the same diffusion coefficient is found in the calculation in Section~\ref{sec:twostates}.
\end{example}

\begin{example}[3 states]\label{ex:threestates}
Now let $M$ be a Markov chain on the triangle with nodes $\s=\{n_1,n_2,n_3\}$ where the state switches with rate $1$ and jumps to the right with probability $1/2+a$ and to the left otherwise (where $|a|\leq 1/2)$. Here $\mu=(\delta_{n_1}+\delta_{n_2}+\delta_{n_3})/3$, $v=[v_1,v_2,v_3]^T$ such that $v_1+v_2+v_3=0$ and
\begin{equation*}
    A=\begin{bmatrix}
        -1 & \frac{1}{2}+a & \frac{1}{2}-a \\
        \frac{1}{2}-a & -1 & \frac{1}{2}+a \\
        \frac{1}{2}+a & \frac{1}{2}-a & -1
    \end{bmatrix}.
\end{equation*}
To find $w$ we solve the linear system $-Aw=v$, which yields
\begin{equation*}
    w = \left[\frac{v_1+(a+1/2)v_2}{3/4 + a^2},\frac{(1/2-a)v_1+v_2}{3/4 + a^2},0\right]^T.
\end{equation*}
This gives
\begin{eqnarray}
    (v,w)&=&\frac{v_1^2+v_1v_2+v_2^2}{3(3/4 + a^2)} = \frac{v_1^2+v_1v_2+v_2^2+v_3(v_1+v_2+v_3)}{3(3/4 + a^2)}\nonumber\\
    &=& \frac{(v_1+v_2+v_3)^2 - (v_1v_2+v_2v_3+v_1v_3)}{9/4 + 3a^2} =-\frac{v_1v_2+v_2v_3+v_1v_3}{9/4 + 3a^2}\label{eq:ex2actpart},
\end{eqnarray}
where we used in the last step that $v_1+v_2+v_3=0$.
Also we compute $\int v^2\dd \mu = (v_1^2+v_2^2+v_3^2)/3$. Now applying Theorem~\ref{thm:limdifcoef} yields
\begin{equation*}
    \lim_{t\rightarrow \infty}\frac{\Var(X_{t})}{t} = 2\kappa + \lambda \int v^2\dd \mu + \frac{2\lambda^2}{\gamma} (v,w) = 2\kappa + \frac{\lambda}{3}(v_1^2+v_2^2+v_3^2) + \frac{2\lambda^2}{\gamma}\frac{(-v_1v_2-v_2v_3-v_1v_3)}{9/4 + 3a^2}.
\end{equation*}
\end{example}

\begin{example}[Ornstein-Uhlenbeck process]\label{ex:OU}
Now let us consider a different kind of example where $M$ is a continuous process, namely an Orstein-Uhlenbeck process satisfying
\begin{equation*}
    \dd M_t = -\theta M_t \dd t + \sigma \dd B_t,
\end{equation*}
where $B_t$ is a Brownian motion independent of everything else (note that a similar process is studied in~\cite{szamel2014self}). This process has stationary distribution $\mu\sim N(0,\sigma^2/(2\theta))$. We take $v(x)=x$ (indeed $\int x \dd \mu = 0$). We know that the generator equals
\begin{equation*}
    A = -\theta x \frac{\dd}{\dd x} + \frac{\sigma^2}{2} \frac{\dd^2}{\dd x^2}
\end{equation*}
and has as domain $D(A)$ all functions in $L^2(\mu)$ of which the first and second (weak) derivative are also in $L^2(\mu)$. A quick inspection shows that if we set $w(x)=x/\theta$, then $w$ in $D(A)$ and $-Aw=v$. Now we compute $(v,w) = \int x^2/\theta \dd \mu = \sigma^2/(2\theta^2)$. Also $\int v^2 \dd \mu = \int x^2\dd \mu = \sigma^2/(2\theta)$.
Now Theorem~\ref{thm:limdifcoef} gives us
\begin{equation*}
    \lim_{t\rightarrow \infty}\frac{\Var(X_{t})}{t} = 2\kappa + \lambda \int v^2\dd \mu + \frac{2\lambda^2}{\gamma} (v,w) = 2\kappa + \frac{\lambda \sigma^2}{2\theta} + \frac{\lambda^2}{\gamma}\frac{\sigma^2}{\theta^2}.
\end{equation*}
Note that the constant $(v,w)=\sigma^2/(2\theta^2)$ could also have been directly obtained by calculating
\begin{equation}
    \Var\left(\lambda \int_0^t v^\gamma_s \dd s\right) = \lambda^2 \int_0^t \int_0^t \Cov(v^\gamma_s,v^\gamma_r)\dd s\dd r
\end{equation}
followed by rescaling and taking limits, since the covariance of the Ornstein-Uhlenbeck process is explicitly known. This yields the same result. Alternatively, one could have used the expression in Remark~\ref{rem:intcov} to see
\begin{equation*}
    (v,-A^{-1}v) = \int_0^\infty \Cov(v_0,v_t)\dd t = \int_0^\infty \frac{\sigma^2}{2\theta}\exp(-\theta t) \dd t = \frac{\sigma^2}{2 \theta^2}.
\end{equation*}
\end{example}

\begin{example}[Sine of Brownian motion with drift]\label{ex:sinBt}
In this example we want the speed process $v_t$ to be $\sin(M_t)$ where $M_t=B_{2at} + bt$, $(B_t,t\geq 0)$ is Brownian motion and $a,b>0$ are constants. However, $X_t$ does not have a stationary (probability) distribution. Therefore we take $M$ to be $B_{2at} + bt$ on a circle $\s$ with radius $1$ and we set $v(\theta)=\sin(\theta)$. Now $\mu = \frac{1}{2\pi}\dd\theta$, so indeed $\int v \dd \mu=0$. The generator is given by
\begin{equation*}
    A = a \frac{\dd^2}{\dd \theta^2} + b \frac{\dd}{\dd \theta}
\end{equation*}
with domain $D(A)$  containing all smooth functions on $\s$. Substituting $w(\theta) = c \sin(\theta) + d \cos(\theta)$ and solving for $c,d$ shows that 
\begin{equation*}
    w(\theta) = \frac{a}{a^2+b^2}\sin(\theta) + \frac{b}{a^2+b^2}\cos(\theta)
\end{equation*}
satisfies $-Aw=v$ with $w\in D(A)$. Now we calculate and see that $\int v^2\dd \mu = \frac{1}{2\pi} \int_0^{2\pi} \sin^2(\theta)\dd \theta=1/2$ and 
\begin{equation*}
    (v,w)=\frac{1}{2\pi} \int_0^{2\pi}  \sin(\theta)\left(\frac{a}{a^2+b^2}\sin(\theta) + \frac{b}{a^2+b^2}\cos(\theta)\right)\dd \theta = \frac{a}{2(a^2+b^2)},
\end{equation*}
so applying Theorem~\ref{thm:limdifcoef}, we see:
\begin{equation}
    \lim_{t\rightarrow \infty}\frac{\Var(X_{t})}{t} = 2\kappa + \lambda \int v^2\dd \mu + \frac{2\lambda^2}{\gamma} (v,w) = 2\kappa + \frac{\lambda}{2} + \frac{2\lambda^2}{\gamma}\frac{a}{a^2+b^2}.\label{eq:difcoefsinbt}
\end{equation}
Note that first of all the last term vanishes when either $a$ or $b$ goes to infinity, similar to what happens when $\gamma$ goes to infinity (see the considerations at the end of Section~\ref{sec:calcdifcoef}). However, note that this part also vanishes when $a$ goes to $0$, even when $b>0$. Indeed, when $a=0$, the speed process is $\sin(M_0+b t)$, where $M_0$ is sampled from $\mu$. Now it is easy to see that $\int_0^t v_s^\gamma\dd s$ is bounded in $t$, so $\var(\int_0^t v_s^\gamma\dd s)/t$ goes to $0$. In that sense the particle is not active in the limit.
\end{example}
\begin{example}\label{ex:multidimOU}
As example for the higher dimensional case, we take $M$ to be the two-dimensional stationary Ornstein-Uhlenbeck process given by
\begin{equation*}
    \dd M_t = -\Theta M_t\dd t + \sigma \dd W_t,
\end{equation*}
where $W_t$ is a two-dimensional Brownian motion,
\begin{equation*}
    \Theta =\begin{bmatrix}
        1 & a \\
        -a & 1 
    \end{bmatrix}
\end{equation*}
and $\sigma,a>0$ are constants. The invariant distribution is $N(0,\sigma^2/2 I)$. We set $v$ to be the identity function.
The corresponding generator is
\begin{equation*}
    Af = - (\nabla f)^T \Theta x + \frac{\sigma^2}{2} \Delta f.
\end{equation*}
First we see that $\Sigma = \frac{\sigma^2}{2}I$.
Now set
\begin{equation*}
    u_1(x) = \frac{1}{1+a^2}(x_1-ax_2), \hspace{1cm} u_2(x) = \frac{1}{1+a^2}(ax_1+x_2),
\end{equation*}
then $-Au_1(x) = x_1=(v)_1(x)$ and $-Au_2(x)=x_2=(v)_2(x)$. Using these we obtain
\begin{equation*}
    ((v)_1,-A^{-1}(v)_1)+((v)_1,-A^{-1}(v)_1) = 2 \frac{1}{1+a^2} (x_1,x_1-ax_2) = \frac{2}{1+a^2}(x_1,x_1) = \frac{\sigma^2}{1+a^2}.
\end{equation*}
Also 
\begin{equation*}
    ((v)_1,-A^{-1}(v)_2)+((v)_2,-A^{-1}(v)_1) = (x_1,ax_1+x_2) + (x_2,x_1-ax_2) = a(x_1,x_1)-a(x_2,x_2) = 0.
\end{equation*}
Here we used that under $\mu$, $(x_1,x_1)=(x_2,x_2)=\sigma^2/2$ and $(x_1,x_2)=(x_2,x_1)=0$. 

Applying Theorem~\ref{thm:ddimcase}, we see that the limiting diffusion matrix equals
\begin{equation}\label{eq:difcoefddim}
    (2\kappa + \lambda \frac{\sigma^2}{2}+ \frac{\lambda^2}{\gamma}\frac{\sigma^2}{1+a^2})I.
\end{equation}
\end{example}

\subsection{Invariance principle}
So far we have calculated the limiting diffusion coefficient of the active particle. In a lot of cases one can in fact show a Central Limit Theorem (CLT) for (the trajectory of) the active particle. This type of problem has been dealt with in a lot of generality under several sets of assumptions before, so we will not provide all the details. 

As we noted before the active particle process decomposes naturally into three parts. First of all, there is the random walk part, which is independent of the rest. The CLT for this case is well-known.

Then there is the martingale part 
\begin{equation*}
    \int_0^t v_s^\gamma \dd\overline N_s.
\end{equation*}
As the name suggests, this term is actually a martingale with respect to the filtration $\mathscr{F}_t=\sigma\{(M_{\gamma s},N_s),0\leq s\leq t\}$ (see Remark~\ref{rem:dynkinmart}). 
Moreover, the active part
\begin{equation*}
    \lambda \int_0^t v_s^\gamma \dd s,
\end{equation*}
is an additive functional of a stationary Markov process and can (under some technical assumptions) be approximated by a martingale with respect to the filtration $\mathscr{F'_t}=\sigma\{M_{\gamma s},0\leq s\leq t\}$ and hence (by independence of $N$ and the active part) also with respect to $\mathscr{F}_t$. This type of result was obtained in~\cite{gordin1978central},~\cite{kipnis1986central},~\cite{toth1986persistent} and~\cite{maxwell2000central}.

Therefore the sum of the martingale part and the active part
\begin{equation*}
    \int_0^t v_s^\gamma \dd N_s
\end{equation*}
can be approximated by a martingale with respect to $\mathscr{F}_t$. Since the martingale part has a source of randomness (the Poisson process $N$) that is independent of the active part, the martingales cannot cancel each other out.
Finally, as is done in the papers that were just cited, one can apply functional martingale central limit theorems such as in~\cite{durrett1978functional} and~\cite{helland1982central} to obtain the CLT for the active particle.

\begin{rmk}\label{rem:dynkinmart}
The fact that the martingale part is actually a martingale with respect to $\mathscr{F}_t$ can be shown from a direct computation. However, this martingale also naturally shows up as a Dynkin martingale. Because of the underlying state process, the position $X_t$ itself is not a Markov process. However, the pair $(X_t,M^\gamma_t)$ is (where $M^\gamma_t$ is $M$ speeded up by a factor $\gamma$). The corresponding generator $L$ is given by
\begin{equation*}
    Lf(x,m) = \lambda (f(x+v(m),m)-f(x,m)) + \gamma (Af(x,\cdot))(m).
\end{equation*} 
Setting $g(x,m)=x$, we see that the following is (formally) a martingale with respect to the natural filtration of $(X_t,M^\gamma_{t})$:
\begin{equation*}
    \mathscr{M}_t := g(X_t,M^\gamma_t)-g(X_0,M^\gamma_0)-\int_0^t Lg(X_s,M^\gamma_s)\dd s = X_t-X_0-\int_0^t \lambda v^\gamma_s\dd s = \int_0^t v^\gamma_s\dd \overline N_s.
\end{equation*}
The quadratic variation of this martingale equals
\begin{equation*}
    \int_0^t (Lg^2 - 2gLg)(X_s,v^\gamma_s) \dd s = \lambda \int_0^t (v^\gamma_s)^2\dd s.
\end{equation*}
Note that by ergodicity of $M$ we have that almost surely
\begin{equation*}
    \lim_{t\rightarrow \infty} \frac{\lambda}{t}\int_0^t (v^\gamma_s)^2\dd s = \lambda \int v^2\dd\mu,
\end{equation*}
which confirms that the martingale part converges to a Brownian motion with diffusion coefficient $\lambda \int v^2\dd\mu$.
\end{rmk}

%% file: reversible.tex
\section{Diffusion coefficient: the role of reversibility}\label{sec:revopt}
Now that we found an expression for the limiting diffusion coefficient of the active particle, we want to understand how it depends on the internal state process. In particular we want to understand the role of reversibility of the internal state process with respect to the stationary measure $\mu$. Recall that we say that the state process $M_t$ is reversible with respect to $\mu$ if the generator $A$ is a self-adjoint operator on its domain in $L^2(\mu)$. We will fix the stationary measure $\mu$ and study processes with this stationary measure. We will also assume in the rest of this section that the internal state space $\s$ is finite, this is mainly to avoid technical complications.

When we inspect the different terms of the diffusion coefficient~\eqref{eq:difcoef}, we see the following.
\begin{itemize}
    \item[a)] The random walk part, $2\kappa$, does not depend on the internal state process.
    \item[b)] The martingale part, $\lambda\int v^2\dd\mu$, only depends on the internal state process through its stationary measure $\mu$.
    \item[c)] The active part, $\tfrac{2\lambda^2}{\gamma}(v,-A^{-1}v)$ depends on the whole internal state process, i.e. its stationary measure as well as its generator.
\end{itemize}
We conclude that given a stationary measure $\mu$, only the active part might depend on the reversibility of the state process with respect to $\mu$. Since also the factor $\tfrac{2\lambda^2}{\gamma}$ is fixed, we will dedicate the rest of this section to studying the behaviour of the term
\begin{equation*}
    \left(v,-A^{-1}v\right).
\end{equation*}

To further specify our results, note that the generator $A$ can be decomposed into a symmetric part $\sym(A)=(A+A^*)/2$ and an antisymmetric part $\asym(A)=(A-A^*)/2$, where $A^*$ denotes the adjoint of $A$ as operators on $L^2(\mu)$. In particular the internal state process is reversible with respect to $\mu$ if $\sym(A)=A$ and accordingly $\asym(A)=0$. We will show the following.
\begin{itemize}
    \item[i)] In Section~\ref{subsec:rev-nonrev} we will consider state processes with the same symmetric part. We will show that the active part of the diffusion coefficient is maximal for the process generated by the symmetric part itself, for any choice of the speed function $v$. In other words: the diffusion coefficient is maximal for the reversible process. Mathematically this means that we will prove that for all $v$ that satisfy $\int v\dd \mu =0$,
    \begin{equation*}
        \left(v,-A^{-1}v\right)\leq \left(v,-\sym(A)^{-1}v\right).
    \end{equation*}
    This is Proposition~\ref{prop:revmax}. We also generalise this to active particles in higher dimensions.
    \item[ii)] In Section~\ref{subsec:revs} we will consider reversible processes with the requirement that the total jumping rate from each point is the same. We will show that in this case there is no reversible process that maximises the diffusion coefficient for each choice of the speed function. In other words: within the class of reversible processes (with the same total jumping rates) there is no optimal reversible process.
\end{itemize}
Before this, we will start with some motivating examples in Section~\ref{subsec:mot}.

\begin{rmk}
Note that the active part of the diffusion coefficient only depends on the ``zero-average"-part of the speed function (see Remark~\ref{rem:zeroavg}). Therefore it remains the same when we replace the speed function $v$ by $v+c$, where $c$ is a constant. Similarly, the active part of the diffusion coefficient is the same for $X^c$ (from Remark~\ref{rem:Xcont}). Because of this, if we replace $v$ by $v+c$ or if we consider the process $X^c$ instead of $X$, the results of this section are still valid.
\end{rmk}

\subsection{Motivation}\label{subsec:mot}
As a motivating example, let us look back at Example~\ref{ex:threestates}. Note that for each $a\in[-1/2,1/2]$, the state process has the same stationary distribution, namely the uniform distribution. However, only for $a=0$ the process is reversible, whereas for $a=1/2$ or $a=-1/2$ the process is completely asymmetric (it only jumps to the right or only to the left, respectively). Hence we can think of $a$ as the parameter that tunes the non-reversibility of the state process. The expression that we found earlier (see~\eqref{eq:ex2actpart}) is
\begin{equation*}
    (v,-A^{-1}v) = \frac{-(v_1v_2+v_2v_3+v_1v_3)}{9/4 + 3a^2}.
\end{equation*}
Since $-(v_1v_2+v_2v_3+v_1v_3)\geq 0$ for $v$ with $\int v\dd \mu = \tfrac{1}{3}(v_1+v_2+v_3) = 0$, this expression is maximal for $a=0$, the reversible case, and decreases like $\tfrac{1}{1+a^2}$ for $a$ away from $0$. We conclude that out of this family of state processes, the reversible process maximizes the diffusion coefficient.

Now for a more general result, we go back to the three states example and note that the symmetric part of the generator (as an operator in $L^2(\mu)$) was the same for each $a$ and the antisymmetric part varied with $a$, indeed:
\begin{equation*}
    \frac{1}{3}\begin{bmatrix}
        -1 & \frac{1}{2}+a & \frac{1}{2}-a \\
        \frac{1}{2}-a & -1 & \frac{1}{2}+a \\
        \frac{1}{2}+a & \frac{1}{2}-a & -1
    \end{bmatrix} =
    \frac{1}{3} \begin{bmatrix}
        -1 & \frac{1}{2} & \frac{1}{2} \\
        \frac{1}{2} & -1 & \frac{1}{2} \\
        \frac{1}{2} & \frac{1}{2} & -1
    \end{bmatrix}
    + \frac{a}{3} \begin{bmatrix}
        0 & 1 & -1 \\
        -1 & 0 & 1 \\
        1 & -1 & 0
    \end{bmatrix}.
\end{equation*}

We want to show that this is true in general: out of all processes (with the same stationary measure $\mu$) of which the symmetric part of the generator is the same, the purely reversible process (so the purely symmetric one) maximizes $(v,-A^{-1}v)$. 

\begin{rmk}
Even though we restrict ourselves in this section to finite state spaces (mainly for technical reasons), notice that the same behaviour (the fact that the diffusion coefficient is maximal for reversible state processes) occurs in Example~\ref{ex:sinBt} and~\ref{ex:multidimOU}.

Indeed, in Example~\ref{ex:sinBt} the state process consists of a reversible part scaled with a constant $a$ and an non-reversible part with constant $b$ (so in particular the process is reversible if and only if $b=0$). The active part of the diffusion coefficient in~\eqref{eq:difcoefsinbt} equals
\begin{equation*}
    \frac{2\lambda^2}{\gamma}\frac{a}{a^2+b^2}.
\end{equation*}
So when we keep $a$ fixed, the active part is maximized in the reversible case.

In Example~\ref{ex:multidimOU} the active part of the diffusion matrix in~\eqref{eq:difcoefddim} equals
\begin{equation*}
    \frac{\lambda^2}{\gamma} \frac{\sigma^2}{1+a^2} I.
\end{equation*}
This matrix is maximal for $a=0$, which is the reversible case.
\end{rmk}

\subsection{Comparing reversible and non-reversible processes}\label{subsec:rev-nonrev}
In order to prove the main result, Proposition~\ref{prop:revmax} below, we first need the following lemma.
\begin{lemma}\label{lemma:skewsym}
Let $C$ be a skew-symmetric matrix. Then both $I+C$ and $I-C^2$ are invertible and for all $w$
\begin{equation*}
    (w,(I+C)^{-1}w) = (w,(I-C^2)^{-1}w) \leq (w,w).
\end{equation*}
\end{lemma}
\begin{proof}
The invertibility of $I+C$ and $I-C^2$ is known, but we repeat it for completeness. Suppose that $I+C$ is not invertible. Then there exists $v\neq 0$ such that $(I+C)v=0$, so $v=-Cv$. Then $(v,v)=-(v,Cv)=0$, which is a contradiction. Similarly if $(I-C^2)v=0$, then $v=C^2v$, so $(v,v)=(v,C^2v)=-(Cv,Cv)\leq 0$, which is a contradiction.

Now let $w$ be arbitrary and set $g=(I-C^2)^{-1}w$ and $h=(I+C)^{-1}w$, which implies that $(I-C)g=h$. Then we see
\begin{equation*}
    (w,(I+C)^{-1}w) = ((I+C)h,h) = (h,h) + (Ch,h)=(h,h)
\end{equation*}
and
\begin{eqnarray*}
    (w,(I-C^2)^{-1}w) &=& ((I-C^2)g,g) = ((I+C)(I-C)g,g) = ((I+C)h,g) \\
    &=& (h,g)+(Ch,g)= (h,g)-(h,Cg) = (h,(I-C)g) = (h,h),
\end{eqnarray*}
which proves the equality.

To prove the inequality, first note that $-C^2$ is positive semidefinite. Therefore the eigenvalues of $I-C^2$ are greater than $1$, so the eigenvalues of $(I-C^2)^{-1}$ are between $0$ and $1$, so $\|(I-C^2)^{-1}\|\leq 1$, which implies that $(w,(I-C^2)^{-1}w) \leq (w,w)$. 
\end{proof}

Since we want to compare a Markov generator with its symmetric part (in $L^2(\mu)$), we recall some properties of this symmetric part. First of all, the symmetric part is again a Markov generator. Moreover, if the original generator has a unique ergodic measure, then the symmetric part generates a reversible process with the same unique ergodic measure. These properties are known, but for the reader's convenience we collect them with a proof in Lemma~\ref{lem:symgen} in the appendix.

Now we can prove the following proposition.
\begin{prop}\label{prop:revmax}
    Let $A$ be the generator of a Markov process on a finite state space with unique ergodic measure $\mu$. Then for all $v$ such that $\int v\dd\mu=0$
    \begin{equation*}
        (v,-A^{-1}v)\leq (v,-\sym(A)^{-1}v),
    \end{equation*}
    where $\sym(A) = (A+A^*)/2$ is the symmetric part of $A$ in $L^2(\mu)$.
    As a consequence, the diffusion coefficient~\eqref{eq:difcoef} is maximized for reversible state processes.
\end{prop}
\begin{proof}
Let $B=(-A+(-A)^*)/2$ be the symmetric part of $-A$ and $D=(-A-(-A)^*)/2$ the skew-symmetric part (in $L^2(\mu)$).
Let $v$ such that $\int v\dd\mu = 0$. Note that $B$ is (strictly) positive definite on the subspace of $w$ such that $\int w\dd\mu=0$, so $B^{-1}$ and $B^{-1/2}$ exist and are symmetric (in $L^2(\mu)$). Now we see
\begin{eqnarray*}
    &&(v,-A^{-1}v)= (v,(B+D)^{-1}v) 
    = (v,(B^{1/2}(I+B^{-1/2}DB^{-1/2})B^{1/2})^{-1}v) \\
    &=& (v,B^{-1/2}(I+B^{-1/2}DB^{-1/2})^{-1}B^{-1/2}v) 
    = (B^{-1/2} v, (I+B^{-1/2}DB^{-1/2})^{-1}B^{-1/2}v).
\end{eqnarray*}
Now write $w=B^{-1/2} v$ and $C=B^{-1/2}DB^{-1/2}$, so 
\begin{equation*}
    (v,-A^{-1}v) = (w,(I+C)^{-1}w).
\end{equation*}
Note that for all $u,u'$
\begin{eqnarray*}
    (u,Cu') &=& (u,B^{-1/2}DB^{-1/2}u') = (B^{-1/2}u,D B^{-1/2}u') \\
    &=& -(DB^{-1/2}u,B^{-1/2}u')=-(B^{-1/2}DB^{-1/2}u,u')=-(Cu,u'),
\end{eqnarray*}
so $C$ is skew-symmetric. Therefore applying Lemma~\ref{lemma:skewsym} gives us that
\begin{eqnarray*}
    (v,-A^{-1}v) &=& (w,(I+C)^{-1}w) \leq (w,w) \\
    &=& (B^{-1/2}v,B^{-1/2}v) = (v,B^{-1}v) = (v,-\sym(A)^{-1}v).
\end{eqnarray*}
\end{proof}

\begin{rmk}
If we assume that $\|B^{-1/2}DB^{-1/2}\|< 1$, we use the Taylor expansion and obtain the more explicit formula:
\begin{equation*}
    (v,-A^{-1}v) = (v,-\sym(A)v) + (w,C^2(I-C^2)^{-1}w),
\end{equation*}
where $w$ and $C$ are as in the proof of Proposition~\ref{prop:revmax}. Indeed in that case
\begin{eqnarray*}
    (w,(I+C)^{-1}w ) &=& \left(w,\sum_{n=0}^\infty (-1)^n C^nw\right) = \sum_{n=0}^\infty (-1)^n (w,C^nw) = \sum_{n=0}^\infty (-1)^{2n} (w,C^{2n}w) \\
    &=& \left(w,\sum_{n=0}^\infty (C^2)^n w\right) = (w,w) + \left(w,\sum_{n=1}^\infty (C^2)^n w\right)\\
    &=& (w,w) + \left(w,C^2\sum_{n=0}^\infty (C^2)^n w\right) = (v,-\sym(A)^{-1}v) + (w,C^2(I-C^2)^{-1}w).
\end{eqnarray*}
Note that in the third equality we used that $C^n$ is skew-symmetric, so $(w,C^nw)=0$ for $n$ odd.
\end{rmk}
Now that we have Proposition~\ref{prop:revmax} for active particles in $\R$, we can use it to generalize to $d$ dimensions. Recall from Theorem~\ref{thm:ddimcase} that the active part of the limiting diffusion matrix of an $\R^d$-valued random walk is $(2\lambda^2/\gamma) D^A$, where
\begin{equation*}
    D^A_{ij} := ((v)_i,-A^{-1}(v)_j)+((v)_j,-A^{-1}(v)_i).
\end{equation*}
The next proposition tells us that in the same context as Proposition~\ref{prop:revmax}, this quantity is optimal for the reversible process.
\begin{cor}\label{cor:revmax}
Let $A$ and $\mu$ be as in Proposition~\ref{prop:revmax}. Then for all $\R^d$-valued $v$ such that $\int v\dd \mu=0$ (in $\R^d$), $D^A$ is dominated by $D^{\sym(A)}$ in the sense that $D^{\sym(A)}-D^A$ is positive definite.
\end{cor}
\begin{proof}
It suffices to show that for all $\alpha\in \R^d$, $\alpha^TD^A\alpha\leq \alpha^TD^{\sym(A)}\alpha$. Let $\alpha\in \R^d$. Then $\alpha \cdot v$ is an $\R$-valued function such that $\int (\alpha \cdot v)\dd\mu = \alpha \cdot (\int v\dd\mu)=0$. Therefore, using Proposition~\ref{prop:revmax}, we see
\begin{eqnarray*}
    \alpha^TD^A\alpha &=& \sum_{i,j=1}^d \alpha_i\alpha_j(((v)_i,-A^{-1}(v)_j)+((v)_j,-A^{-1}(v)_i)) = 2 ((\alpha\cdot v),-A^{-1}(\alpha\cdot v))\\
    &\leq& 2 ((\alpha\cdot v),-\sym(A)^{-1}(\alpha\cdot v)) = \alpha^TD^{\sym(A)}\alpha.
\end{eqnarray*}
\end{proof}

\subsection{Comparing reversible processes}\label{subsec:revs}
Proposition~\ref{prop:revmax} tells us that among all generators with the same symmetric part, the symmetric part itself maximizes the diffusion coefficient of the active particle. Now one might wonder whether there are classes of reversible internal state processes that yield the same diffusion coefficient for each speed function $v$. The following lemma shows us that this is not the case.
\begin{lemma}
Let $A$ and $B$ be Markov generators with reversible measure $\mu$. Suppose that for every $v$ with $\int v\dd\mu=0$, $(v,-A^{-1}v)=(v,-B^{-1}v)$. Then $A=B$.
\end{lemma}
\begin{proof}
Define the following linear subspaces of $L^2(\mu)$:  $V_\mu:=\{v| \int v\dd\mu = 0\}$ and $V_1 = \{c\1|c\in \R\}$. Note that $V_\mu$ and $V_1$ are orthogonal in $L^2(\mu)$ and in fact $V_\mu$ is the orthogonal complement of $V_1$ in $L^2(\mu)$, so the action on $V_\mu$ and $V_1$ together fully define $A$ and $B$. Also note that $A$ and $B$ are $0$ on $V_1$ and are invertible when restricting to $V_\mu\rightarrow V_\mu$. It suffices to show that $A$ and $B$ are equal on $V_\mu$, so in turn it suffices to show that $A^{-1}$ and $B^{-1}$ are equal on $V_\mu$. For this let $v,w\in V_\mu$. Then
\begin{eqnarray*}
    (v,-A^{-1}w) &=& \frac{1}{2}((v+w),-A^{-1}(v+w))-(v,-A^{-1}v)-(w,-A^{-1}w)) \\
    &=& \frac{1}{2}((v+w),-B^{-1}(v+w))-(v,-B^{-1}v)-(w,-B^{-1}w)) = (v,-B^{-1}w).
\end{eqnarray*}
This shows that $A^{-1}=B^{-1}$ on $V_\mu$, so we conclude that $A=B$.
\end{proof}

Now that we know that different reversible processes cannot yield the same diffusion coefficients, it could still be that certain reversible processes yield larger diffusion coefficients than others. To answer this question, we need to normalise in some way. Otherwise if we replace the generator $A$ by $cA$ for some constant $c>1$, the diffusion coefficient is divided by that constant $c$, so $A$ trivially yields larger diffusion coefficients than $cA$. We normalise here by comparing reversible processes that have the same total jumping rate from each point.
The next lemma tells us that in that case no process strictly dominates all the others, it depends on the speed function $v$.
\begin{lemma}
Let $A$ and $B$ be Markov generators on a finite state space that are reversible with respect to $\mu$. Additionally assume that the total jump rate from each state is the same for $A$ and $B$. Then either $A=B$ or there exist $v,w\in V_\mu$ such that
\begin{equation*}
    (v,-A^{-1}v)>(v,-B^{-1}v)\hspace{0.5cm}\text{and}\hspace{0.5cm}(w,-A^{-1}w)<(w,-B^{-1}w).
\end{equation*}
\end{lemma}
\begin{proof}
Let $A$ and $B$ be as stated. Now assume that there are no $v,w\in V_\mu$ such that $(v,-A^{-1}v)>(v,-B^{-1}v)$ and $(w,-A^{-1}w)<(w,-B^{-1}w)$. 
Without loss of generality assume that for all $v\in V_\mu$, $(v,-A^{-1}v)\geq (v,-B^{-1}v)$. 
This implies that $-A^{-1}\geq -B^{-1}$ (in the sense that $-A^{-1}-(-B^{-1})$ is symmetric and positive definite on $V_\mu$). With the fact that $-A,-B$ are positive definite, this in turn implies that $-B\geq -A$, so $A - B\geq 0$ on $V_\mu$. Since also $Av=Bv=0$ for $v\in V_1$, this implies that $A-B\geq 0$ on $L^2(\mu)$. Now if we define $D$ to be the diagonal matrix with $D_{ii}=\mu_i$, then $D(A-B)\geq 0$ and symmetric with respect to the usual inner product in $\R^d$. Also, $A-B$ and (hence) $D(A-B)$ have zeroes on the diagonal (because of the equal jump rates), so the trace of $D(A-B)$ is $0$. Therefore the eigenvalues of $D(A-B)$ are non-negative and sum to $0$, so they are all $0$. This implies that $D(A-B)=0$, so $A=B$.
\end{proof}

%% file: largedev.tex
\section{Large deviations}\label{sec:largedev}
In this section we derive a large deviation principle (LDP) for $\tfrac{1}{t}X_t$.
\footnote{For the definition of the Large Deviation Principle and for Varadhan's lemma and the Gaertner-Ellis theorem, see for intance~\cite{dembo1998zeitouni} or~\cite{den2008large}.} 
The active particle that we are studying is what is called a slow-fast system in the literature and a lot of research has already been done about its large deviations. Because of this it is not our goal here to present this result in the highest possible generality. We would rather see which formulas are obtained and study their behaviour, in particular the relation between the rate function and the reversibility of $M$. Therefore we reduce (as in Section~\ref{sec:revopt}) to the case where the state space $\s$ of $M$ is finite (and hence where $(v^\gamma,s\geq 0)$ is bounded).

\begin{rmk}
Note that we don't need anywhere in this section that $\int v\dd\mu=0$.
\end{rmk}

Since we will express the rate function for $\tfrac{1}{t}X_t$ in terms of the rate function of the empirical process corresponding to the underlying state process, we quickly recall some results that we will use. We write
\begin{equation*}
    \chi_t = \frac{1}{t}\int_0^t \delta_{M_s}\dd s
\end{equation*}
and denote by $P_t$ the distribution of $\chi_t$ in the space of probability measures on $\s$. Then we know from \cite{donsker1975asymptotic} that $(P_t,t\geq 0)$ satisfies an LDP with good rate function $I_e$ given by
\begin{equation}\label{eq:rateGeneral}
    I_e(\xi) = \sup_{u>0} \left(-\sum_{i=1}^n\xi_i \frac{(Au)_i}{u_i}\right).
\end{equation}
In case $A$ is symmetric, this reduces to
\begin{equation}\label{eq:ratesym}
    I_e(\xi) = (u,-Au),
\end{equation}
where $u_i = \sqrt{\xi_i/\mu_i}$ (note that we assumed that $\mu$ has full support, so $\mu_i>0$ for all $i$) and the inner product is (as usual) with respect to $\mu$.

\subsection{Large deviations rate function}
To obtain the large deviations rate function of $X_t/t$, we start by calculating the logarithmic moment generating function (log-mgf) of $X_t$: $F_t(\alpha) = \log \E\left[\e^{\alpha X_t}\right]$ for $\alpha\in \R^d$. To calculate it we first observe that by independence of $Y$ and the rest,
\begin{equation}\label{eq:logmgf}
    F_t(\alpha) = \log \E\left[\e^{\alpha \left(\sqrt{2\kappa}Y_t + \int_0^tv_s^\gamma\dd N_s\right)}\right]
    = \log \E\left[\exp(\alpha \sqrt{2\kappa}Y_t)\right] + \log \E\left[\exp\left(\alpha \int_0^tv_s^\gamma\dd N_s\right)\right].
\end{equation}
The first term is just the log-mgf of a simple random walk speeded up with a factor $2\kappa$. Therefore at time $t$ it equals the difference of two Poisson random variables with parameter $\kappa t$, so we obtain that 
\begin{equation*}
    \log \E\left[\exp(\alpha Y_{2\kappa t})\right] = \log (\exp(\kappa t (\e^{\alpha}-1))\exp(\kappa t (\e^{-\alpha}-1))) = 2\kappa t (\cosh(\alpha)-1).
\end{equation*}
To calculate the second term, we first condition on $v^\gamma=(v^\gamma_s,0\leq s\leq t)$ and obtain
\begin{eqnarray*}
    &&\E\left[\left.\exp\left(\alpha \int_0^tv_s^\gamma\dd N_s\right)\right|v^\gamma\right] = \lim_{n\rightarrow \infty} \E\left[\left. \exp\left(\alpha \sum_{i=0}^{n-1}v^\gamma_{s_i}(N_{s_{i+1}}-N_{s_i})\right)\right|v^\gamma\right] \\
    &=& \lim_{n\rightarrow \infty} \prod_{i=0}^{n-1} \E\left[ \exp\left(\alpha v^\gamma_{s_i}(N_{s_{i+1}}-N_{s_i})\right)|v^\gamma\right] = \lim_{n\rightarrow \infty} \prod_{i=0}^{n-1} \exp\left(\lambda\left(\e^{\alpha v_{s_i}^\gamma}-1\right)(s_{i+1}-s_i)\right) \\
    &=& \lim_{n\rightarrow \infty} \exp\left(\sum_{i=0}^{n-1} \lambda \left(\e^{\alpha v_{s_i}^\gamma}-1\right)(s_{i+1}-s_i))\right) = \exp\left(\lambda \int_0^t \left(\e^{\alpha v_{s}^\gamma}-1\right)\dd s\right).
\end{eqnarray*}
Therefore we see that the second term of~\eqref{eq:logmgf} equals
\begin{equation*}
    \log\E \exp\left(\lambda \int_0^t \left(\e^{\alpha v_{s}^\gamma}-1\right)\dd s\right).
\end{equation*}
We conclude that
\begin{equation}\label{eq:logmgfend}
    F_t(\alpha) = 2\kappa t (\cosh(\alpha)-1) + \log\E \exp\left(\lambda \int_0^t \left(\e^{\alpha v_{s}^\gamma}-1\right)\dd s\right).
\end{equation}

Now we can compute the large deviation free energy function $F(\alpha)$ as the limit of $F_t(\alpha)/t$.
We see for the first term that
\begin{equation}\label{eq:freeenergyfirst}
    \lim_{t\rightarrow \infty} \frac{2\kappa t (\cosh(\alpha)-1)}{t} = 2\kappa (\cosh(\alpha)-1).
\end{equation}
Now for the second term define $h_\alpha$ as a function on measures on $\s$ given by
\begin{equation*}
    h^\gamma_\alpha(\xi) = \frac{\lambda}{\gamma} \int_\s \left(\e^{\alpha v(x)}-1\right)\xi(\dd x).
\end{equation*}
This enables us to rewrite the second part of $F_t(\alpha)$ and use Varadhan's lemma to obtain
\begin{eqnarray*}
    &&\lim_{t\rightarrow\infty}\frac{1}{t} \log \E \exp\left(\frac{\lambda}{\gamma} \int_0^{\gamma t} \left(\e^{\alpha v(M_s)}-1\right) \dd s  \right) \\
    &=& \lim_{t\rightarrow\infty}\frac{1}{t} \log \E \exp\left(t\gamma \frac{\lambda}{\gamma} \int_\s \left(\e^{\alpha v(x)}-1\right) \left(\frac{1}{\gamma t}\int_0^{\gamma t} \delta_{M_s} \dd s\right)(\dd x)  \right) \\
    &=& \gamma \lim_{t\rightarrow\infty} 
    \frac{1}{\gamma t} \log \E \exp\left( \gamma t h^\gamma_\alpha(\chi_{\gamma t})\right) = \gamma \sup_{\xi} (h^\gamma_\alpha(\xi) - I_e(\xi)).
\end{eqnarray*}
Note the latter equals 
\begin{equation}\label{eq:freeenergysecond}
    \sup_{\xi}(\lambda(\varphi_\xi(\alpha)-1) - \gamma I_e(\xi)),
\end{equation}
where $\varphi_\xi(\alpha)=\int_\s\exp(\alpha v(x))\xi(\dd x)$ denotes the mgf of $v$ under $\xi$ evaluated at $\alpha$. Taking together~\eqref{eq:freeenergyfirst} and~\eqref{eq:freeenergysecond}, we conclude that
\begin{equation}\label{eq:freeenergy}
    F(\alpha) = \lim_{t\rightarrow\infty}\frac{F_t(\alpha)}{t} = 2\kappa (\cosh(\alpha)-1) + \sup_{\xi}(\lambda(\varphi_\xi(\alpha)-1) - \gamma I_e(\xi)).
\end{equation}
Using the Gaertner-Ellis theorem, we now obtain the large deviation principle for $X_t/t$ with rate function given by the Legendre transform of $F(\alpha)$:
\begin{equation*}
    I(x) = \sup_{\alpha}(\alpha x - F(\alpha)) = \sup_{\alpha}(\alpha x - 2\kappa (\cosh(\alpha)-1) - \sup_{\xi}(\lambda(\varphi_\xi(\alpha)-1) - \gamma I_e(\xi))).
\end{equation*}

\begin{rmk}
A very similar computation shows that a similar expression holds in the multidimensional case. Indeed, if we set $F_t(\alpha) = \log \E \exp(\alpha\indot X_t)$ for $\alpha\in \R^d$, we obtain
\begin{equation*}
    F(\alpha)=\lim_{t\rightarrow\infty}\frac{F_t(\alpha)}{t} = 2\kappa \sum_{i=1}^d (\cosh(\alpha_i)-1) + \sup_{\xi}(\lambda(\varphi_\xi(\alpha)-1) - \gamma I_e(\xi)),
\end{equation*}
where 
\begin{equation*}
    \varphi_\xi(\alpha)=\int_\s\e^{\alpha \indot v(x)}\xi(\dd x).
\end{equation*}
Then again we can take the Legendre transform to find the rate function $I$.
\end{rmk}

\begin{example}\label{ex:twostatesLDP}
We return to Example~\ref{ex:twostates} to obtain an explicit expression for the large deviations free energy function. Note that the state process is reversible with respect to the stationary measure $\mu=(1/2,1/2)$. Using~\eqref{eq:ratesym}, fixing a probability measure $\xi$ on $\{\pm1\}$ and setting $u_i=\sqrt{(\xi_i/(1/2)}=\sqrt{2\xi_i}$, we see
\begin{equation*}
    I_e(\xi) = (u,-Au) = \frac{1}{2}(\sqrt{2\xi_1}-\sqrt{2\xi_{-1}})^2 =  (\sqrt{\xi_1}-\sqrt{\xi_{-1}})^2 = 1 - 2\sqrt{\xi_1\xi_{-1}}.
\end{equation*}
Parametrizing $\xi=(r,1-r)$, we see 
\begin{eqnarray*}
    \sup_{\xi}(\lambda(\varphi_\xi(\alpha)-1) - \gamma I_e(\xi)) &=& \sup_{0\leq r\leq 1} (\lambda(r \e^\alpha + (1-r)\e^{-\alpha} - 1) -\gamma (1-2\sqrt{r(1-r})) \\
    &=& \lambda(\e^{-\alpha}-1) - \gamma + \sup_{0\leq r\leq 1} (2\lambda \sinh(\alpha)r +2\gamma \sqrt{r(1-r)}).
\end{eqnarray*}
A simple calculation shows that the latter equals
\begin{equation*}
    \lambda(\e^{-\alpha}-1) - \gamma + \sqrt{\gamma^2+\lambda^2\sinh^2(\alpha)} + \lambda \sinh(\alpha) = \lambda(\cosh(\alpha)-1) + \sqrt{\gamma^2+\lambda^2\sinh^2(\alpha)} - \gamma,
\end{equation*}
so with~\eqref{eq:freeenergy}, we see
\begin{equation}\label{eq:twostatesfreeenergy}
    F(\alpha) = (2\kappa + \gamma) (\cosh(\alpha)-1) +  \sqrt{\gamma^2+\lambda^2\sinh^2(\alpha)} - \gamma.
\end{equation}

\end{example}

\begin{rmk}
In the case of $X^c$ (from Remark~\ref{rem:Xcont}), the calculations become a bit easier. Instead of the symmetric random walk $Y_{2\kappa t}$ we directly work with the continuous limit $B_{2\kappa t}$. But more importantly, there is no additional randomness from the Poisson process $N$. Following the analogous computations for this part, we find the same results with $\varphi_\xi(\alpha)$ replaced by $\alpha \int v(x) \xi(\dd x).$ 

In this section we worked with a finite state space, so all the computations and quantities here are well-defined. However, for a more general state process, for the original process $X$ one would need 
\begin{equation*}
    \E \exp\left(\lambda \int_0^t \left(\e^{\alpha v_{s}^\gamma}-1\right)\dd s\right)<\infty
\end{equation*}
to get a finite free energy. Setting $t\ll 1$, this implies that we need something like
\begin{equation*}
    \E \e^{\e^{v_0}}<\infty,
\end{equation*}
which is a very strong assumption that for instance for the Ornstein-Uhlenbeck process is not satisfied. 

Changing to $X^c$ means getting rid of the Poisson jumps, which takes away one of the exponentials. So we expect that an LDP holds for a lot more state processes in the $X^c$ case than for the original process $X$.
\end{rmk}

\subsection{The role of reversibility}
Our goal now is to show a result that is similar to Proposition~\ref{prop:revmax}. Indeed, we show that if an active particle has a state process generated by some generator $A$, then the rate function of this active particle is greater (pointwise) than the rate function of the active particle of which the state process is generated by the symmetric part of $A$. In other words: a reversible state process yields a lower rate function. Before we show this, we will prove the following lemma about a similar result for the rate functions of the empirical measures corresponding to the state processes.

\begin{lemma}
Let $A$ be a Markov generator with unique ergodic measure $\mu$ and let $\sym(A)$ be its symmetric part (in $L^2(\mu)$). Denote the rate functions of the corresponding empirical processes by $I_e^A$ and $I_e^{\sym(A)}$, respectively. Then for all probability measures $\xi$, $I_e^{\sym(A)}(\xi)\leq I_e^A(\xi)$.
\end{lemma}
\begin{proof}
Let $\xi$ be a probability measure on $\s$. We set $u_i=\sqrt{\xi_i/\mu_i}\geq 0$. Also define for $m\in \N$, $u^m_i = u_i$ if $u_i>0$ and $u^m_i=1/m$ otherwise. Note that $u^m_i>0$ for all $i$ and that $u^m\rightarrow u$ in $L^2(\mu)$ (since it converges pointwise and $\s$ is finite). Finally note that $\xi_i/u^m_i = \mu_i u_i$ for all $i$. Now, using~\eqref{eq:rateGeneral}, we see that for all $m$
\begin{equation*}
    I^A_e(\xi) = \sup_{u'>0}-\sum_i^n\xi_i \frac{(Au')_i}{u'_i} \geq -\sum_i^n\xi_i \frac{(Au^m)_i}{u^m_i} = -\sum_{i=1}^n \mu_i u_i (Au^m)_i = (u,-Au^m).
\end{equation*}
Therefore, using~\eqref{eq:ratesym}, we conclude
\begin{equation*}
    I^A_e(\xi)\geq \lim_{m\rightarrow\infty} (u,-Au^m) = (u,-Au) = (u,-\sym(A)u) = I_e^{\sym(A)}(\xi).
\end{equation*}
\end{proof}
Now we use this to prove the following result.
\begin{cor}\label{cor:revmaxDV}
Let $A$ be a Markov generator with unique ergodic measure $\mu$ and let $\sym(A)$ be its symmetric part (in $L^2(\mu)$). Denote the rate functions of the corresponding active particle processes by $I^A$ and $I^{\sym(A)}$ and the free energy functions of those processes by $F^A$ and $F^{\sym(A)}$, respectively. Then for all $\alpha\in \R: F^A(\alpha)\leq F^{\sym(A)}(\alpha)$ and for all $x\in \R: I^{\sym(A)}(x)\leq I^A(x)$.
\end{cor}
\begin{proof}
Since for all $\xi$, $I_e^{\sym(A)}(\xi)\leq I_e^A(\xi)$, it follows that for all $\alpha$, \begin{equation*}
    \sup_{\xi}(\lambda(\varphi_\xi(\alpha)-1) - \gamma I^{\sym(A)}_e(\xi)) \geq \sup_{\xi}(\lambda(\varphi_\xi(\alpha)-1) - \gamma I^A_e(\xi)),
\end{equation*}
so $F^{\sym(A)}(\alpha)\geq F^A(\alpha)$. Since this holds for all $\alpha$, similary it follows that for all $x$, $I^{\sym(A)}(x)\leq I^A(x)$.
\end{proof}
\begin{rmk}
Note that in the case that $F(\alpha)$ is sufficiently smooth, the limiting diffusion coefficient (or matrix, in the higher dimensional case) is given by the second derivative, or, more generally, the Hessian of $F(\alpha)$ in $0$. By Corollary~\ref{cor:revmaxDV}, the free energy function is dominated by the free energy function of the active particle with state process generated by the symmetric part pointwise everywhere and they are equal for $\alpha=0$. Therefore we see in that case that the Hessian at $0$ (and therefore the limiting diffusion matrix) is dominated by the Hessian of the symmetric version. This is consistent with the results of Proposition~\ref{prop:revmax} and Corollary~\ref{cor:revmax}.
\end{rmk}

%% file: twostates.tex
\section{The \texorpdfstring{$2$}{2}-state case: explicit formulas}\label{sec:twostates}
In the case where there are just two states, we can compute a lot of things explicitly with different methods. Therefore this section is dedicated to the active particle with two states. In this case the active particle has a position $x\in\Z$ and a velocity $v\in \{-1,1\}$.
The process $\{(X_t,v_t): t\geq 0\}$ is described via the generator
\beq\label{gena}
Lf(x,v) &=&  \lambda (f(x+v,v)- f(x,v)) \nonumber\\
&+& \kappa (f(x+1, v)+ f(x-1,v) -2f(x,v))
\nonumber\\
&+ & \gamma(f(x,-v)-f(x,v))
\eeq
This is interpreted as follows: with rate $\lambda$ the process makes a jump in the direction of the velocity,
with rate $\kappa$ it makes a random walk jump and with rate $\gamma$ it flips velocity $v\to -v$.
If we denote $\mu(x,t,v)$ the probability to be at location $x\in\Z$ with velocity $v\in \{-1,1\}$ at time $t>0$, the generator
\eqref{gena} corresponds to the master equation (or Kolmogorov forward equation)
\beq\label{mastereqi}
\frac{d\mu(x,t,v)}{dt} &=& \lambda \mu(x-v,t,v) + \kappa (\mu(x-1,t, v)+ \mu(x+1,t,v)) + \gamma \mu(x,t,-v)
\nonumber\\
&-& (2\kappa+\lambda + \gamma) \mu(x,t,v).
\eeq

\subsection{The Fourier Laplace transform of the distribution}
The master equation \eqref{mastereqi} can be solved using Fourier-Laplace transform.
We define
\begin{equation*}
    \hat{\mu} (q,t,v)= \sum_{x} e^{iqx} \mu(x,t,v)
\end{equation*}
and view this quantity as a two-column, denoted $\overline{\mu}(q,t, \cdot)$ indexed by row index $v=1,-1$. The master equation \eqref{mastereqi} then becomes, after Fourier transform:
\be\label{bonk}
\frac{d}{dt} \overline{\mu}(q,t) = M(q) \overline{\mu}(q,t)
\ee
with $M(q)$ a symmetric two by two matrix of the form
\beq\label{M}
&&M(q)=
\left(
\begin{array}{cc}
a & b\\
b & a^*
\end{array}
\right)
\eeq
where $*$ denotes complex conjugate and where
\begin{eqnarray}\label{abb}
    a &=& (2\kappa +\lambda)(\cos(q)-1)-\gamma + i\lambda\sin(q) \nonumber\\
    b &=& \gamma.
\end{eqnarray}

For the analysis of the scaling behavior of the position of the particle, it is convenient to further Laplace transform $\omu (q,t)$
i.e., we define, for $z>0$ the column vector
\be\label{lapma}
\homu(q,z)= \int_0^\infty \omu(q,t) e^{-zt}\ dt.
\ee
Then, from \eqref{bonk} we find
\[
\homu(q,z) = ( zI- M(q) )^{-1} {\bar{\mu}}_0 (q).
\]
For the initial position and velocity we choose $X_0=0$, and $v=\pm 1$ with probability $1/2$.
Then we have, ${\bar{\mu}}_0 (q)= \frac12 (1,1)^T$ where $T$ denotes transposition.
We further define the Fourier Laplace transform of the distribution of the  particle position:
\begin{equation*}
    S(q,z)=\int_0^\infty \E e^{iqX_t} e^{-zt} \ dt = \sum_v \homu(q,z,v)=(1,1) \homu(q,z).
\end{equation*}

Then we have, using \eqref{lapma}
\begin{equation*}
    S(q,z)= (\homu(q,z,1)) + (\homu(q,z, -1)) =
\frac12 (1,1) (zI-M(q))^{-1} (1,1)^T.
\end{equation*}

Using the explicit formulas \eqref{M}, \eqref{abb}, we obtain
\be\label{sq}
S(q,z)= \frac{2\gamma + z - (\lambda +2\kappa) (\cos(q)-1)}{(\gamma +z -(\lambda+2\kappa) (\cos(q)-1))^2-\gamma^2 +  \lambda^2 \sin^2(q)}.
\ee
For a more general velocity distribution at time zero, i.e., $X_0=0$, and $v=1$, resp.\ $v=-1$, with probability $\alpha$, resp. $1-\alpha$, we find
\begin{equation*}
    S(q,z)= \frac{i\lambda (2\alpha -1) \sin(q)+ 2\gamma + z - (\lambda +2\kappa) (\cos(q)-1)}{(\gamma +z -(\lambda+2\kappa) (\cos(q)-1))^2-\gamma^2 +  \lambda^2 \sin^2(q)}.
\end{equation*}

\subsection{The limiting diffusion coefficient}
We can now use the explicit formula \eqref{sq} to obtain
the limit distribution of
$\epsi X_{\epsi^{-2} t}$ as $\epsi\to 0$. This amounts to
understand the scaling behavior of $\epsi^2 S(\epsi q, \epsi^2 z)$.
In particular  $\epsi X_{\epsi^{-2} t}\to \caN(0,\si^2 t)$ as $\epsi\to 0$ (in distribution), where $\caN(0,\si^2 t)$
denotes a normal with mean zero and variance $\si^2 t$, corresponds to the limiting
scaling behavior
\[
\lim_{\epsi\to 0} \epsi^2 S(\epsi q, \epsi^2 z)=  \frac{1}{z + \frac{q^2}{2}\si^2}.
\]
If we obtain this scaling behavior, we call $\si^2$ the (limiting) diffusion constant.
Indeed, we compute from the exact formula \eqref{sq}
\begin{equation*}
    \lim_{\epsi\to 0} \epsi^2 S(\epsi q, \epsi^2 z) = \frac{1}{z + \frac{q^2}{2}\si^2}
\end{equation*}
with the limiting diffusion constant
\be\label{difcon}
\si^2= 2\kappa +\lambda + \frac{\lambda^2}{\gamma}.
\ee
This is consistent with the limiting diffusion coefficient that we obtained in Example~\ref{ex:twostates}.
\subsection{Moment generating function and large deviations}
We choose the starting point $X_0=0$ and with random initial velocity, i.e., $v=\pm 1$ with probability $1/2$.
This allows us to compute the moment generating function via
\beq\label{momgen}
\E( e^{\alpha X_t}) & = & \frac12(1,1)e^{t M(-i \alpha)} (1,1)^T.
\eeq
This amounts to compute the exponential of the matrix $M(q)$ from \eqref{M} which can be done using
diagonalization, and results in
\begin{equation*}
    e^{t M(q)}= \frac{e^{tA}}{2\gamma B} G(t,q)
\end{equation*}
where $G(t,q)$ is given by the symmetric two by two matrix
\begin{equation*}
    G(t,q)=
\left(
\begin{array}{cc}
A_{11} & A_{12}\\
A_{12} & A_{11}^*
\end{array}
\right)
\end{equation*}
where
\begin{eqnarray*}
    A_{11} &=& -2\gamma\lambda i \sin(k) \sinh (Bt) + 2\gamma B \cosh (tB) \\
    A_{12} &=& 2\gamma^2 \sinh(tB)
\end{eqnarray*}
and where
\begin{eqnarray*}
    A &=& (\cos(k)-1)(2\kappa + \lambda) -\gamma \\
    B &=& \sqrt{\gamma^2-\lambda^2 \sin^2(k)}.
\end{eqnarray*}

Moreover, we see from~\eqref{momgen} that the free energy function
\be
F(\alpha) = \lim_{t\to\infty} \frac1t\log\E \left(e^{\alpha X_t}\right)\nonumber\\
\ee
is equal to the largest eigenvalue of the symmetric matrix $M(-i\alpha)$, which is explicitly given by
\begin{equation*}
    M(-i\alpha)=
\left(
\begin{array}{cc}
(2\kappa+\lambda) (\cosh(\alpha)-1) + \lambda\sinh(\alpha) -\gamma & \gamma\\
\gamma & (2\kappa+\lambda) (\cosh(\alpha)-1) - \lambda\sinh(\alpha)-\gamma
\end{array}
\right).
\end{equation*}
This gives
\begin{equation}\label{free}
F(\alpha)= (2\kappa +\lambda)(\cosh (\alpha) -1)
+   \sqrt{\gamma^2 + \lambda^2\sinh^2(\alpha)}-\gamma,
\end{equation}
which agrees with~\eqref{eq:twostatesfreeenergy}.

Let us look at three relevant limiting cases for the ``free energy function'' $F$ from \eqref{free}.
\ben
\item[a)] Expanding the free energy function $F$ around $\alpha\approx 0$ gives
\[
F(\alpha)= \frac12D\alpha^2 + O(\alpha^4)
\]
with $D= 2\kappa + \lambda + \tfrac{\lambda^2}{\gamma}$.
This is consistent with the diffusion constant found in Example~\ref{ex:twostates} and in~\eqref{difcon}.
The function $F(\alpha)$ in \eqref{free} can be analytically extended in a neighborhood of
the origin in the complex plane, and as a consequence, we can reobtain the central limit theorem (which we found via the scaling behavior of the characteristic function) from the large deviation
free energy, see \cite{bryc1993remark}.
\item[b)] In the limit $\gamma\to\infty$ the free energy function
becomes
\[
F(\alpha)= (\cosh(\alpha)-1)(2\kappa + \lambda)
\]
which corresponds to the large deviations of a symmetric random walk jumping
with rates $\kappa +\lambda/2$ to the right or left. This is indeed the (slow-fast) scaling limit of the process as we saw before. For large values of $\gamma$ we have
\[
F(\alpha)= (\cosh(\alpha)-1)(2\kappa + \lambda) + \frac{\lambda^2}{2\gamma} \sinh^2(\alpha)  + o(1/\gamma)
\]
Remark also that $F$ in \eqref{free} is non-increasing as a function of $\gamma$.
\item[c)] In the continuum limit we rescale $\lambda\to \epsi\lambda$, $\gamma\to \epsi^2\gamma$, $X_t\to \epsi X_{\epsi^{-2} t}$, we find
\be\label{dublim}
\lim_{\epsi\to 0}\lim_{t\to\infty}\frac1{t}\log  \E^{\epsi\lambda, \epsi^2\gamma}\left( e^{\alpha \epsi X_{\epsi^{-2} t}}\right) = \kappa\alpha^2 + \sqrt{\gamma^2+ \lambda^2 \alpha^2}-\gamma^2
\ee
which corresponds to the large deviation free energy of the continuum model (see also \cite{pietzonka2016extreme}), i.e., the limits $\epsi\to 0$ and
$t\to \infty$ in \eqref{dublim} commute.
\een

%% file: appendix.tex
\section{Appendix}
\subsection{Integral approximation in \texorpdfstring{$L^2(\p)$}{L2}}
Firs we show how it follows from Assumption~\eqref{eq:assumptioncontL2} that the integral~\eqref{eq:Xtdefinition} is well-defined. An alternative more abstract way to establish the well-definedness of this integral is as follows. From Assumption~\eqref{eq:assumptioncontL2} follows that $v^\gamma_s$ admits a c\`adl\`ag version and hence the integral can be interpreted as an ordinary Riemann-Stieltjes integral of a c\`adl\`ag function against an integrator of bounded variation.
\begin{lemma}\label{lemma:L2approxintegral}
Let $W_s$ be $N_s$, $\overline N_s$ or $\lambda s$. Then Assumption~\eqref{eq:assumptioncontL2} implies that
\begin{equation}
    \lim_{n\rightarrow \infty} \sum_{i=1}^n v^\gamma_{s_i}(W_{s_{i+1}}-W_{s_i}) =: \int_0^t v^\gamma_s\dd W_s,
\end{equation}
exists as a limit in $L^2(\p)$ and the $s_i=s_i^n$ are partitions of $[0,t]$ of which the mesh sizes go to $0$.
\end{lemma}
\begin{proof}
Without loss of generality, set $\gamma=1$. By linearity it suffices to let $W$ be either $\overline N$ or $W_s=s$. By the completeness of $L^2(\p)$, it suffices to show that for each $\epsilon>0$ there exists $\delta>0$ such that if the meshes of two partitions are under $\delta$, the $L^2$-distance between the corresponding Riemann sums is smaller than $\epsilon$. Indeed, this implies both that for each sequence of partitions (with mesh going to $0$), the Riemann sums form a Cauchy sequence in $L^2(\p)$ (and hence has a limit) and that every sequence of partitions yields the same limit. 

Let $\epsilon>0$. Choose $\delta>0$ such that for all $0\leq s,s'\leq t$ with $|s-s'|<\delta$, $\E[(v_s-v_{s'})^2]\leq \epsilon$.

Now let $s^1=(s^1_i)_{i=0}^n$ and $s^2=(s^2_i)_{i=0}^m$ denote partitions of $[0,t]$. Assume that $s^1$ is such that $\mesh(s^1)\leq \delta$ and that $s^2$ is a refinement of $s^1$. Denote by $s^1_{i*}$ the largest partition element of $s^1$ that is smaller than or equal to $s^2_i$ and note that for all $i$, $|s^2_i-s^1_{i*}|\leq \mesh(s^1)\leq\delta$. In particular for all $i,j$, using Cauchy-Schwarz,
\begin{equation*}
    |\E[(v_{s^2_i} - v_{s^1_{i*}})(v_{s^2_j} - v_{s^1_{j*}})]| = |\Cov(v_{s^2_i} - v_{s^1_{i*}},v_{s^2_j} - v_{s^1_{j*}})|\leq \sqrt{\epsilon^2}=\epsilon.
\end{equation*}
Now
\begin{equation*}
    \sum_{i=0}^{m-1} v_{s^2_i}(W_{s^2_{i+1}}-W_{s^2_i})-\sum_{i=1}^n v_{s^1_i}(W_{s^1_{i+1}}-W_{s^1_i}) = 
    \sum_{i=0}^{m-1} (v_{s^2_i}-v_{s^1_{i*}})(W_{s^2_{i+1}}-W_{s^2_i}).
\end{equation*}
Therefore,
\begin{eqnarray}
    &&\E\left(\sum_{i=0}^{m-1} v_{s^2_i}(W_{s^2_{i+1}}-W_{s^2_i})-\sum_{i=1}^n v_{s^1_i}(W_{s^1_{i+1}}-W_{s^1_i})\right)^2 \nonumber\\
    &=& \sum_{i,j=0}^{m-1} \E \left[(v_{s^2_i}-v_{s^1_{i*}})(W_{s^2_{i+1}}-W_{s^2_i})(v_{s^2_j}-v_{s^1_{j*}})(W_{s^2_{j+1}}-W_{s^2_j})\right].\label{eq:expr1}
\end{eqnarray}
In case $W=\overline N$, using the independence of $v$ and $N$ and the fact that $\overline N$ has increments with expectation $0$, \eqref{eq:expr1} equals
\begin{eqnarray*}
    \sum_{i=0}^{m-1} \E (v_{s^2_i}-v_{s^1_{i*}})^2 \E(W_{s^2_{i+1}}-W_{s^2_i})^2 \leq \sum_{i=0}^{m-1} \epsilon \lambda ({s^2_{i+1}}-{s^2_i}) = \lambda t\epsilon
\end{eqnarray*}
In the case $W_s=s$, \eqref{eq:expr1} equals
\begin{equation*}
    \sum_{i,j=0}^{m-1}\E \left[(v_{s^2_i}-v_{s^1_{i*}})(v_{s^2_j}-v_{s^1_{j*}})\right](s^2_{i+1}-s^2_i)(s^2_{j+1}-s^2_j) \leq \sum_{i,j=0}^{m-1} \epsilon (s^2_{i+1}-s^2_i)(s^2_{j+1}-s^2_j) = \lambda t^2 \epsilon.
\end{equation*}
Set $C=\max(\lambda t,t^2)$. Then in both cases
\begin{equation*}
    \E\left(\sum_{i=0}^{m-1} v_{s^2_i}(W_{s^2_{i+1}}-W_{s^2_i})-\sum_{i=1}^n v_{s^1_i}(W_{s^1_{i+1}}-W_{s^1_i})\right)^2\leq C\epsilon,
\end{equation*}
so the $L^2$ distance between the Riemann sums is smaller than $\sqrt{C\epsilon}$. Now if $s^1$ and $s^2$ are any partitions of $[0,t]$ with mesh smaller than $\delta$ (so if it is not necessarily the case that one is a refinement of the other), let $s^3$ be a refinement of both. Then $\mesh(s^3)\leq\delta$, so the Riemann sum corresponding to $s^3$ is close to the Riemann sums of both $s^1$ and $s^2$. Now the triangle inequality gives the desired result.
\end{proof}

\subsection{Properties of the symmetric part of a Markov generator}

Since we want to compare a Markov generator with its symmetric part, we will introduce this symmetric part and show some of its relevant properties in the next lemma. In particular, the symmetric part is again a Markov generator.
\begin{lemma}\label{lem:symgen}
Let $A$ be a Markov generator with unique ergodic measure $\mu$ and denote by $A^*$ its adjoint in $L^2(\mu)$. Assume that $\mu$ has full support. Then $\sym(A):=(A+A^*)/2$ is also a Markov generator with unique ergodic measure $\mu$. Moreover $\sym(A)$ is reversible with respect to $\mu$.
\end{lemma}
\begin{proof}
Denote by $e^i$ the $i^{\text{th}}$ unit vector. Then for every matrix $B$
\begin{equation*}
    (e^i,Be^j) = \sum_{k=1}^ne^i_k (Be^j)_k \mu_k = B_{ij}\mu_k. 
\end{equation*}
Therefore
\begin{equation}\label{eq:symAind}
    \sym(A)_{ij} = \frac{1}{\mu_i} (e^i,\sym(A)e^j) = \frac{1}{2\mu_i}((e^i,Ae^j)+(e^j,Ae^i)) = \frac{1}{2\mu_i}(\mu_i A_{ij}+\mu_jA_{ji}).
\end{equation}
Since $A_{ij}\geq 0$ for all $i\neq j$, \eqref{eq:symAind} implies that $\sym(A)_{ij}\geq 0$ for $i\neq j$. Also, setting $i=j$, we obtain $\sym(A)_{ii}=A_{ii}\leq 0$.

Since $\mu$ is a stationary distribution for $A$, we know that for all $v$,
\begin{equation*}
    (A^*\1,v) = (\1,Av) = \int Av\dd \mu = 0,
\end{equation*}
so $A^*\1=0$. Therefore
\begin{equation*}
    \sym(A)\1 = \frac{1}{2}(A\1 + A^*\1) = 0.
\end{equation*}
We conclude (so far) that $\sym(A)$ has negative diagonal elements, positive off-diagonal elements and $0$ row sums, so $\sym(A)$ is a Markov generator.

By construction, $\sym(A)$ is self-adjoint in $L^2(\mu)$, so $\sym(A)$ is reversible with respect to $\mu$ and in particular $\mu$ is stationary for $\sym(A)$.

Finally, since $A$ has a unique ergodic measure with full support, (the Markov processes generated by) it is irreducible. Since (by~\eqref{eq:symAind}) for all $i,j$, $A_{ij}\neq 0 \implies \sym(A)_{ij}\neq 0$, this implies that also (the Markov process generated by) $\sym(A)$ is irreducible. Therefore it can have at most one invariant measure, which implies that $\mu$ is the unique invariant measure and hence the unique ergodic measure.
\end{proof}